\documentclass[12pt,twoside]{article}
\usepackage{etex}
\usepackage[dvipsnames]{xcolor} 
\usepackage{color}
\usepackage{booktabs}
\usepackage{amsthm}
\usepackage{amsfonts,amssymb,amsxtra,url,float} 
\allowdisplaybreaks[4] 
\usepackage[colorlinks,
  linkcolor=magenta, %
  anchorcolor=Periwinkle,
  citecolor=violet,
  urlcolor=blue
  ]{hyperref} 
\usepackage{enumitem}
\usepackage{geometry} 
\usepackage{rotating} 
\usepackage{lscape} 
\usepackage{multirow}
\usepackage{graphicx} 
\usepackage{subfigure} 
\usepackage{tikz}
\usepackage{pgfplots}
\usepackage{tikz-3dplot}
\usetikzlibrary{patterns}
\usetikzlibrary{3d,calc}
\usetikzlibrary{decorations.pathreplacing,decorations.markings}
 \tikzset{
  on each segment/.style={
    decorate,
    decoration={
      show path construction,
      moveto code={},
      lineto code={
        \path [#1]
        (\tikzinputsegmentfirst) -- (\tikzinputsegmentlast);
      },
      curveto code={
        \path [#1] (\tikzinputsegmentfirst)
        .. controls
        (\tikzinputsegmentsupporta) and (\tikzinputsegmentsupportb)
        ..
        (\tikzinputsegmentlast);
      },
      closepath code={
        \path [#1]
        (\tikzinputsegmentfirst) -- (\tikzinputsegmentlast);
      },
    },
  },
  mid arrow/.style={postaction={decorate,decoration={
        markings,
        mark=at position 0.6 with {\arrow[#1]{stealth}} 
      }}},
}
\usetikzlibrary{arrows}
\usetikzlibrary{trees}

\usetikzlibrary{matrix}
\usetikzlibrary{patterns}
\usetikzlibrary{shadings} 
\usepackage{fancyhdr} 
\def\headertitle{Normed representations of weight quivers}
\def\fstpage{1} 
\def\page{$\begin{matrix} {\color{white}0} \\ \thepage \end{matrix}$} 

\usepackage[all]{xy} 
\usepackage{dsfont} 
\usepackage{cite}
\usepackage{mathrsfs} 
\numberwithin{figure}{section}
\usepackage{marginnote} 
\usepackage{graphicx} 
\usepackage{multicol} 

\usepackage{enumitem}
\setenumerate[1]{itemsep=0pt,partopsep=0pt,parsep=\parskip,topsep=3pt}
\setitemize[1]{itemsep=0pt,partopsep=0pt,parsep=\parskip,topsep=3pt}
\setdescription{itemsep=0pt,partopsep=0pt,parsep=\parskip,topsep=3pt}
\setlist[itemize]{leftmargin=35pt}
\setlist[enumerate]{leftmargin=35pt}
\usepackage{changepage} 

\newtheorem{theorem}{Theorem}[section]
\newtheorem{lemma}[theorem]{Lemma}
\newtheorem{corollary}[theorem]{Corollary}
\newtheorem{main theorem}[theorem]{Main Theorem}
\newtheorem{proposition}[theorem]{Proposition}
\newtheorem{definition}[theorem]{Definition}

\newtheorem{remark}[theorem]{Remark}
\newtheorem{example}[theorem]{Example}
\newtheorem{notation}[theorem]{Notation}
\newtheorem{question}[theorem]{Question}

\newtheorem{assumption}[theorem]{Assumption}

\usetikzlibrary{arrows}

\numberwithin{equation}{section}

\def\orcid{
\begin{tikzpicture}[baseline=-1mm]
\filldraw[Green!35] (0,0) circle (5pt);
\filldraw[white] (0,0) node{\tiny\textbf{iD}};
\end{tikzpicture}
}

\def\NN{\mathbb{N}} 
\def\ZZ{\mathbb{Z}} 
\def\RR{\mathbb{R}} 
\def\CC{\mathbb{C}} 
\def\II{\mathbb{I}}
\def\FF{\mathbb{F}}
\def\EE{\mathbb{E}}

\newcommand{\Ker}{\operatorname{Ker}}
\newcommand{\Ima}{\operatorname{Im}}
\newcommand{\Pic}{Figure\ }
\newcommand{\modcat}{\mathsf{mod}}
\newcommand{\Modcat}{\mathsf{Mod}}
\newcommand{\repcat}{\mathsf{rep}}
\newcommand{\Repcat}{\mathsf{Rep}}

\newcommand{\kk}{{l\hspace{-1mm}k}} 
\newcommand{\Q}{\mathcal{Q}} 
\def\I{\mathcal{I}}
\def\fcts{\mathfrak{s}}
\def\fctt{\mathfrak{t}}
\newcommand{\J}{\mathcal{J}}
\newcommand{\e}{\varepsilon}
\newcommand{\Hom}{\mathrm{Hom}} %
\newcommand{\End}{\mathrm{End}} %
\newcommand{\rad}{\mathrm{rad}}
\newcommand{\op}{\mathrm{op}}
\newcommand{\w}[1]{\widehat{#1}}

\def\alg{\mathit{\Lambda}}
\def\Func{\mathrm{Func}}
\def\bfS{\mathbf{S}}
\def\Nor{\mathsf{Nor}}
\def\Ban{\mathsf{Ban}}
\def\ave{\mathfrak{A}}

\def\heart{{\color{red}\pmb{\heartsuit}}}

\def\compos{\ \lower-0.2ex\hbox{\tikz\draw (0pt, 0pt) circle (.1em);} \ }

\newcommand{\defines}{\it\color{red}}
\title{\bf Normed representations of weight quivers
}

\vspace{5mm}

\author{
Yu-Zhe Liu$^{\ref{Author1}, \href{https://orcid.org/0009-0005-1110-386X}{\orcid}\ref{orcid1},~\ref{CorrespondingAuthor}}$
}
\date{ }
\begin{document}






\maketitle

\begin{enumerate}[label=\textbf{\color{red}\arabic*}] \footnotesize

  \item
    \begin{center}
      School of Mathematics and Statistics, Guizhou University,
      Guiyang 550025, Guizhou, China;

      E-mail:  \url{liuyz@gzu.edu.cn} / \url{yzliu3@163.com}
    \end{center} \label{Author1}
\end{enumerate}
\vspace{2mm}
\begin{enumerate}[label=\textbf{\color{red}$\dag$}]
  \item \footnotesize
    \begin{center}
      Corresponding author
    \end{center} \label{CorrespondingAuthor}
\end{enumerate}
\vspace{2mm}
\begin{enumerate} \footnotesize
    \item[{\orcid}] \centering  ORCID: \href{https://orcid.org/0009-0005-1110-386X}{0009-0005-1110-386X}
      \label{orcid1}
\end{enumerate}

\vspace{1mm}

\begin{adjustwidth}{1cm}{1cm}
  \noindent \footnotesize
  \textbf{Abstract}: Let $A$ and $B$ be two tensor rings given by weight quivers.
    We introduce norms for tensor rings and $(A,B)$-bimodules, and define an important category $\mathscr{A}^p_{\varsigma}$ in this paper whose object is a triple $(N,v,\delta)$ given by an $(A,B)$-bimodule $N$, a special element $v\in V$ satisfying some special conditions, and a special $(A,B)$-homomorphism $\delta: N^{\oplus_p 2^{\dim A}} \to N$ and each morphism $(N,v,\delta) \to (N',v',\delta')$ is given by an $(A,B)$-homomorphism $\theta: N\to N'$ such that $\theta(v)=v'$ and $\delta' \theta^{\oplus 2^{\dim A}} = \theta\delta$ hold.
    We show that $\mathscr{A}^p_{\varsigma}$ has an initial object such that Daniell integration, Bochner integration, Lebesgue integration, Stone--Weierstrass Approximation Theorem, power series expansion, and Fourier series expansion are morphisms in $\mathscr{A}^p_{\varsigma}$ starting with this initial object.

\vspace{1mm}


     \label{2020MSC}

\vspace{1mm}

     \label{Keywords}
\end{adjustwidth}

\tableofcontents


\def\la{\langle} 
\def\ra{\rangle} 
\def\lala{\langle\!\langle}
\def\rara{\rangle\!\rangle}
\def\=<{\leqslant}
\def\>={\geqslant}
\def\vecd{\pmb{d}}
\def\vecg{\pmb{g}}
\def\spa{\mathrm{span}}
\def\Gal{\mathrm{Gal}}
\def\Mat{\mathbf{Mat}}
\def\bfE{\pmb{E}}
\def\op{\mathrm{op}}
\def\norm{\mathfrak{n}}
\def\bigboxplus{\mathop{ \raisebox{-0.3em}{ \scalebox{1.75}{\text{\(\boxplus\)}} } }}

\def\scrN{\mathscr{N}\!or}
\def\scrA{\mathscr{A}}
\def\Pc{\mathds{P}}

\def\invlim{\underleftarrow{\lim}}
\def\dirlim{\underrightarrow{\lim}}

\def\dimA{d_A}
\def\dimB{d_B}
\def\homo{\varsigma}

\def\dd{\mathrm{d}}
\def\Bochner{\mathrm{B}}
\def\Lebesgue{\mathrm{L}}

\newcommand{\basis}[1]{\mathfrak{B}_{#1}}

\section{Introduction}

There has been a growing interest in the algebraic characterization of analysis in recent times. For instance, \cite{BCS2006diff, BCS2015diff, Lemay2019, APL2021diff, IL2023ana, Lemay2023} investigated the categorical descriptions of differential, and \cite{Bax1960,Sag1965ana,Rot1969-1,Rot1969-2, CL2018Cart-int, CL2018int} explored the categorical/algebraic descriptions of integrations.
Lebesgue integration was formulated by Henri Lebesgue as a generalization of Riemann integration \cite{L1928} in 1902. It has being extensively used in numerous fields of analysis.

In \cite{Lei2023FA}, Leinster offered a method to characterize Lebesgue integration and $L_p$-spaces by using a specific category $\scrA^p$ ($p\>= 1$). More precisely, Lebesgue integration can be conceptualized as a morphism $T: L_p([0,1]) \to \FF$ equipped with a juxtaposition map $\gamma: L_p([0,1]) \oplus L_p([0,1]) \to L_p([0,1]) $ and an average map $\ave: \FF \oplus \FF$, $(x_1,x_2)\mapsto \frac{x_1+x_2}{2}$. The categorification of integration has garnered scholarly attention for an extended period, resulting in the establishment of integral categories, cf. \cite{Sag1965ana, CL2018Cart-int, CL2018int}.
Moreover, Rota--Baxter algebra \cite{Bax1960,Rot1969-1,Rot1969-2} provides another algebraic description of integration, which is also recognized as a leading area of research in algebra.

Normed modules originally denoted vector spaces over a field equipped with a norm, primarily used for analysis of function spaces, cf. \cite[etc]{Guo2013, GP2020}.
A primary objective of this paper is to offer a categorical description of generalized $L_p$-space $\w{\bfS_{\homo}(\II_A)}$ in which integrable functions are precisely $f: \II_A \to B$. Furthermore, we provide a categorical description for abstract integrations by using a morphism $\w{T}$ originates from an initial object in $\mathscr{A}^p_{\varsigma}$, and $\w{T}$ satisfies the axiomatic definition of Daniell integration given in \cite{Dan1918}.
In summary, our primary focus to investigate the following question.

\begin{question} \rm \label{Quest:1}
Let $A$ and $B$ be two finite-dimensional $\kk$-algebras, $f:A\to B$ be a function, and $X$ be a subset of $A$.
\begin{enumerate}[label=(Q\arabic*)]
  \item Under what conditions is $f|_X$ integrable? \label{Q1}
  \item If $f|_X$ is integrable, then what is its integral? \label{Q2}
  \item For a vector space $V$, if it is a normable vector space, then we can define integration in many cases. Is the definition of integration unique? \label{Q3}
\end{enumerate}
\end{question}

One of the main purposes of this article is to answer Question \ref{Quest:1}. In \cite{Lei2023FA}, $L_p([0,1])$ is a vector space over $\RR$ whose $\RR$-action is defined as $\RR \times L_p([0,1]) \to L_p([0,1])$, $(r,f) \mapsto (rf: x\mapsto rf(x))$.
In \cite{LLHZ2025}, authors provided a categorification $\w{\bfS_{\homo}(\II_A)}$ of $L_p([0,1])$ whose elements are integrable functions $f:\II_A \to \kk$ and showed that $\w{\bfS_{\homo}(\II_A)}$ is a left $A$-module with the left $A$-action $A \times \w{\bfS_{\homo}(\II_A)} \to \w{\bfS_{\homo}(\II_A)}$, $(a,f)\mapsto (a.f: x \mapsto a.f(x) := \homo(a)f(x))$ ($a\in A$, $f(x)\in \kk$).
The definition $a.f(x):= \homo(a)f(x)$ indicates that setting a homomorphism $\homo: A \to \kk$ of algebras is necessary.
To answer Question \ref{Quest:1}, we need to define the action of $a\in A$ on a function $f: X\to B$ defined on $X\subseteq A$. It follows that the set of integrable functions, still written as $\w{\bfS_{\homo}(\II_A)}$, may be an $(A, B)$-bimodule,
and we need a homomorphism $\homo: A \to B$ between two finite-dimensional algebras.
Meanwhile, an important perspective is that the normed module $L_p([0,1])$ in \cite{Lei2023FA} and the normed module $\w{\bfS_{\homo}(\II_A)}$ in \cite{LLHZ2025} are seen as a normed $(\kk,\kk)$-bimodule and a normed $(A,\kk)$-bimodule, respectively.
Thus, we extend the definition of normed module in this paper, especially providing a norm to $\w{\bfS_{\homo}(\II_A)}$. To do this, we have \textbf{three difficulties} that have not been encountered in references \cite{Lei2023FA, LLHZ2025,LLL2025}:
\begin{itemize}
  \item In the case of $\homo:A\to B$, what is the $(A,B)$-bimodule structure and norm for $\w{\bfS_{\homo}(\II_A)}$?
  \item Why does $\w{\bfS_{\homo}(\II_A)}$, as a bimodule, need a $(A,B)$-homomorphism $\Pc: B^{\times I} \to\w{\bfS_{\homo}(\II_A)}$ in \ref{Np2} such that $\Pc((1)_{1\times I}) = (\mathbf{1}_{\II_A}:\II_A \to \{1_B\})$?
  \item How do $\FF$-isomorphisms in the Galois groups act on some elements lying in an extension of $\FF$ during the proof process of certain key conclusions (such as Lemma \ref{lemm:A})?
\end{itemize}

Setting $\FF$ a base field in this paper, and for any algebra $\alg$, we use $1_{\alg}$ and $0_{\alg}$ to present the identity and zero in $\alg$. The paper is organized as follows.
Second \ref{sect:Perlim} is about some basic knowledge, mainly reviewing tensor rings, which are a more general class of finite-dimensional algebras than quiver algebras. The algebras used in this article are all tensor rings, so we can obtain more general results than \cite{Lei2023FA, LLHZ2025}.
In Section \ref{sect:norm}, we introduce norms for tensor rings and the representation of weight quiver.
Given a homomorphism between two tensor rings $A$ and $B$. We introduce two categories $\scrN^p_{\homo}$ and $\scrA^p_{\homo}$ in Section \ref{sect:two cats}, where $\scrN^p_{\homo}$ is a category whose objects are normed $(A,B)$-bimodules with some conditions (see three conditions \ref{Np1}, \ref{Np2}, and \ref{Np3} given in Definition \ref{def:Np}) and whose morphisms are special $(A,B)$-homomorphisms, and $\scrA^p_{\homo}$ is a full subcategory of $\scrN^p_{\homo}$ whose objects are Banach $(A,B)$-bimodules.
In algebraic convention, objects in $\scrN^p_{\homo}$ and $\scrA^p_{\homo}$ are defined as triples.
In this Section, we provide the first result of this paper as follows.

\begin{theorem}[{\rm Theorem \ref{thm:1}}] \label{thmA}
The category $\scrA^p_{\homo}$ has an initial objects which is a triple $(\w{\bfS_{\homo}(\II_A)}, \mathbf{1}, \w{\gamma}_{\xi})$ of the completion $\w{\bfS_{\homo}(\II_A)}$ of the $(A,B)$-bimodule $\bfS_{\homo}(\II_A)$, function $\mathbf{\II_A}: \II_A \to \{1_B\}$, and a  juxtaposition map such that \ref{Np1}, \ref{Np2}, and \ref{Np3} hold.
\end{theorem}

\noindent
Furthermore, we obtain the second main result of this paper.

\begin{theorem}[{\rm Theorem \ref{thm:2}}] \label{thmB}
Assume that $\FF$ is a field with a field extension $\FF/\RR$, and $\FF$, $A$ and $B$ are completed.
Then the triple $(\bfS_{\homo}(\II_A), \mathbf{1}_{\II_A}$, $\gamma_{\xi})$ in $\scrN^p_{\homo}$ is an $\scrA^p_{\homo}$-initial object. Thus, there is a unique morphism $h:(\bfS_{\homo}(\II_A),\mathbf{1}_{\II_A},\gamma_{\xi}) \to (N,v,\delta)$ in $\scrN^p_{\homo}$, such that the diagram
\[ \xymatrix@C=1.5cm{
  (\bfS_{\homo}(\II_A),\mathbf{1}_{\II_A},\gamma_{\xi})
  \ar[r]^{h} \ar[d]_{\subseteq}
& (N,v,\delta) \\
 (\w{\bfS_{\homo}(\II_A)},\mathbf{1}_{\II_A},\w{\gamma}_{\xi})
  \ar[ru]_{\w{h}}
 &
} \]
commutes. Here, $\w{h}$ is an $(A,B)$-homomorphism induced by the completion $\w{\bfS_{\homo}(\II_A)}$ of $\bfS_{\homo}(\II_A)$, and it is an extension of $h$.
\end{theorem}

\noindent
Sections \ref{sect:app-int} and \ref{sect:app-approx} are two applications.
The third main result of this paper provides a categorical description of abstract integration which satisfies three conditions (see \ref{I1}, \ref{I2}, and \ref{I3} given in Section \ref{sect:app-int}). See the following theorem.

\begin{theorem} \label{thmC}
Assume that $\FF$ is a field with a field extension $\FF/\RR$, and $\FF$, $A$ and $B$ are completed.
\begin{itemize}
\item[\rm(1)] {\rm (Proposition \ref{prop:A})}
  The category $\scrA^p_{\homo}$ contains an object which is of the form $(B$, $\mu_{\II_A}(\II_A), \ave)$.
  Here, $\mu_{\II_A}$ is a measure, and $\ave$ is a map $B^{\oplus 2^{\dim_{\FF}A}} \to A$ sending each element $(b_1,b_2,\ldots, b_{2^{\dim_{\FF}A}})$ to a weighted average.

\item[\rm(2)] {\rm (Theorem \ref{thm:3})}
There exists a unique morphism $T: (\bfS_{\homo}(\II_A), \mathbf{1}_{\II_A}, \gamma_{\xi}) \to (B, \mu_{\II_A}(\II_A)1_B, \ave)$ in ${\scrN^1_{\homo}}$ such that
\[ \xymatrix@C=1.5cm{
  (\bfS_{\homo}(\II_A),\mathbf{1}_{\II_A},\gamma_{\xi})
  \ar[r]^{T} \ar[d]_{\subseteq}
& (B, \mu_{\II_A}(\II_A)1_B, \ave) \\
  (\w{\bfS_{\homo}(\II_A)},\mathbf{1}_{\II_A},\w{\gamma}_{\xi})
  \ar[ru]_{\w{T}}
 &
} \]
commutes. Here, $\w{T}$ is an $(A,B)$-homomorphism in $\scrA^p_{\homo}$ induced by the completion $\w{\bfS_{\homo}(\II_A)}$ of $\bfS_{\homo}(\II_A)$. It can be written as $\displaystyle (\scrA^1_{\homo})\int_{\II_A}(\cdot)\dd\mu_{\II_A}$ in the case of $p=1$.
Furthermore, if $p=1$, then we have the following results:
\begin{itemize}
  \item[\rm(a)] $\w{T}$ sends each function $f = \sum_i b_i \mathbf{1}_{I_i} \in \bfS_{\homo}(\II_A)$ $(\forall i\ne j, I_i \cap I_j = \varnothing$, and $\II_A = \bigcup_i I_i)$ to an element $\sum_i b_i \mu_{\II_A}(I_i)$;
  \item[\rm(b)] $\w{T}$ is an $(A,B)$-homomorphism between two $(A,B)$-bimodules;
  \item[\rm(c)] for each $f\in\w{\bfS_{\homo}(\II_A)}$, we have
    \[ (\scrA^1_{\homo})\int_{\II_A} \Vert f\Vert \dd\mu_{\II_A}
    = \omega 1_B \in \RR^{\>=0} 1_B := \{r1_B \mid r\in\RR^{\>=0}\} ~(\subseteq B) \]
    where $\Vert f\Vert$ is the function $\Vert f \Vert : \II_A \to B, x \mapsto \Vert f(x)\Vert_{B, p}$, and $\Vert \cdot\Vert_{B, p}$ is a norm defined on $B$;
  \item[\rm(d)] for each nonincreasing Cauchy sequence $\{f_n\}_{n\in\NN}$ in $\w{\bfS_{\homo}(\II_A)}$ with $\invlim f_n = 0: \II_A \to \{0_B\}$, we have
    \[ \invlim (\scrA^1_{\homo}) \int_{\II_A} f_n \dd\mu_{\II_A}
    = 0_B
    = (\scrA^1_{\homo}) \int_{\II_A} \invlim f_n \dd\mu_{\II_A}. \]
\end{itemize}
\end{itemize}
\end{theorem}
\noindent
All functions $f$ in $\w{\bfS_{\homo}(\II_A)}$ are integrable functions, and their integrals are written as $\displaystyle (\scrA^1_{\homo})\int_{\II_A} f \dd\mu_{\II_A}$ in this paper. The reason why we need to prove (a), (b), and (c) in the above theorem is due to the axiomatic definition of the Daniell integral given in \cite{Royden1988}. Thus, we answered Question \ref{Quest:1} \ref{Q1} and \ref{Q2} by the above theorem.
Combine Theorems \ref{thmB} and \ref{thmC}, we have answered \ref{Quest:1} \ref{Q3} by using the uniqueness of $\w{T}$.
In Section \ref{sect:app-approx}, we provide a categorical description of the Stone--Weierstrass Approximation Theorem, see Corollary \ref{coro:SWThm}. Finally, we consider some examples in Section \ref{sect:exp}.

\section{Preliminaries} \label{sect:Perlim}

We recall some concepts about tensor rings in this section. These concepts can be found in references \cite[Section 2.1]{BTLS2023}, and which all originate from \cite[Section 7.1]{Gab2006Repr} (or refer to \cite{Gab1972,Gab1973}), \cite[Section 10]{DR1975}, \cite[Section 1B]{Gre1975}, \cite[Section 2]{R1976}, \cite[Section 2]{Sim1979}, \cite[Section 2]{BTCB2024}, \cite[Sections 2 and 3]{GLF2016} and \cite[Section 2]{Geu2017}.

\subsection{Weight quivers and tensor rings}

First, we recall the definitions of weight quiver and tensor ring given in \cite{LFZ2016, GLF2016, GLF2020}.

\begin{definition}[Weight quivers and modulations]  \rm \label{def:weightquiver} ~
\begin{itemize}
  \item[(1)] \cite[Definition 2.2]{LFZ2016} A {\defines weight quiver} is a pair $(\Q,\vecd)$ given by a quiver and a $\NN_+$-vector $\vecd = (d_i)_{i\in \Q_0} \in \NN_+^{\times \Q_0}$. Here, $\vecd$ is called a {\defines wight} of $(\Q,\vecd)$.

  \item[(2)] \cite[Remark 4.1]{LFZ2016} Let $\FF$ be a field, an {\defines $\FF$-modulation} of a weight quiver $(\Q,\vecd)$ is a pair
$((D_i)_{i\in \Q_0}, (A_{\alpha})_{\alpha\in\Q_1})$ given by two sequences $(D_i)_{i\in \Q_0}$ and $(A_{\alpha})_{\alpha\in\Q_1})$, where
  \begin{enumerate}[label=(2.\arabic*)]
    \item each $D_i$ is a finite-dimensional division $\FF$-algebra with $\dim_{\FF}D_i = d_i$;
      \label{(2.1)}
    \item $A_{\alpha}$ is a $(D_{\fcts(\alpha)},D_{\fctt(\alpha)})$-bimodule, i.e., is both a left $D_{\fcts(\alpha)}$-module and a right $D_{\fctt(\alpha)}$-module;
      \label{(2.2)}
    \item and the action of $\FF$ on $A_{\alpha}$ is central
      (i.e., $\forall f\in \FF$ and $x\in  A_{\alpha}$, $f.x=x.f$).
      \label{(2.3)}
  \end{enumerate}
\end{itemize}
\end{definition}

\begin{definition}[Tensor rings] \rm \label{def:tensorring}
Let $R$ be a ring with identity and $A = {_RA_R}$ be an $(R,R)$-bimodule.
\begin{itemize}
  \item[(1)] Recall that a {\defines tensor ring} is the direct sum
\[ R \la A\ra := \bigoplus_{n\>=0} A^{\otimes_R n} = R \oplus A \oplus (A\otimes_R A) \oplus (A\otimes_R A \otimes_R A) \oplus \cdots \]
whose multiplication is defined by the natural $R$-balanced map
\[ A^{\otimes_R m} \times A^{\otimes_R n} \to A^{\otimes_R (m+n)}, (x,y) \mapsto x\otimes y. \]
  \item[(2)] Furthermore, a {\defines complete tensor ring} is the direct sum
\[ R \lala A\rara
 := \prod_{n\>=1} A^{\otimes_R n}
= \mathop{\underleftarrow{\lim}}\limits_{l\in\NN}
  \bigg(
    R\langle A\rangle\bigg/\bigoplus_{n\>= l}
      A^{\otimes_R n}
  \bigg). \]
\end{itemize}
\end{definition}

\begin{remark} \rm
Tensor rings are called the path algebras of $(\Q,\vecd,\vecg)$ in \cite[Definition 4.2]{LFZ2016}, \cite[Definition 3.5]{GLF2016} and \cite{GLF2020}.
\end{remark}

The following shows that each algebra $\kk\Q/\I$ given by a bound quiver $(\Q,\I)$ is a tensor ring.
Here, $\I$ is an ideal of the path algebra $\kk\Q$ of the quiver $\Q$.

Assume $\Q=(\Q_0,\Q_1,\fcts,\fctt)$, where $\Q_0$ and $\Q_1$ respectively are vertex set and arrow set
and $\fcts$ and $\fctt$ are functions $\Q_1\to\Q_0$ respectively send each arrow to its starting point and ending point.
We define the multiplication of two paths $\wp_1$ and $\wp_2$ is the composition $\wp_1\wp_2$
if $\fctt(\wp_1)=\fcts(\wp_2)$, cf. \cite[Chap II]{ASS2006}.
Then for a field $\kk$, we have $\kk\Q$ is a tensor ring by the isomorphism
\[\kk\Q \cong \bigoplus_{n\>=0} (\kk\Q_1)^{\otimes_R n} = R\langle \kk\Q_1\rangle. \]
Here, $R = \spa_{\kk}(\Q_0) = \prod_{v\in\Q_0}\kk\e_v$
($\e_v$ is the path of length zero corresponded by the vertex $v$),
$\kk\Q_1$ is the $\kk$-vector space generated by the set $\Q_1$,
and $(\kk\Q_1)^{\otimes_R n}\cong\kk\Q_n$ is isomorphic to
the $\kk$-vector space $\kk\Q_n$ generated by all paths of length $n$~
\footnote{In particular, $(\kk\Q_1)^{\otimes_R 0} = \kk\Q_0$,
and $\Q_n$ is the set of all path of length $n$. }.
The natural $\kk$-balanced map
\[ (\kk\Q_1)^{\otimes_R m} \times (\kk\Q_1)^{\otimes_R n} \to (\kk\Q_1)^{\otimes_R (m+n)} \]
is given by the multiplication
\[ \Q_m \times \Q_n \mapsto \Q_{m+n},
(\wp_1,\wp_2) \mapsto
\begin{cases}
  \wp_1\wp_2, & \fctt(\wp_1)=\fcts(\wp_2); \\
  0, & \fctt(\wp_1)\ne \fcts(\wp_2)
\end{cases} \]
of paths on a quiver.
Furthermore, one can check that each quiver algebra
\[ \kk\Q/\I = \bigoplus_{n\>=0} (\kk\Q_1+\I)^{\otimes_R n} = R\langle \kk\Q_1+\I \rangle \]
is a tensor ring. In this sense, each path $\wp=a_1a_2\cdots a_{\ell}$
($a_1$, $a_2$, $\ldots$, $a_{\ell}$ $\in \Q_1$) of length $\ell$
is an element $a_1\otimes a_2 \otimes \cdots \otimes a_{\ell}$
in the $\kk$-module $(\Q_1+\I)^{\otimes_R \ell}$.
If $\alg=\kk\Q/\I$ is a {\defines finite-dimensional algebra}, i.e., the dimension $\dim_{\kk}\alg$ of $\alg$ is finite,
then there exists $N\in\NN$ such that $(\kk\Q_1+\I)^{\otimes_R n}=0$ holds for all $n\>=N$.
Thus, $\alg$ is a complete tensor ring since
\[ R\la \kk\Q_1+\I \ra
= \bigoplus_{n\>=0} (\kk\Q_1+\I)^{\otimes_R n}
\cong \prod_{n\>=0} (\kk\Q_1+\I)^{\otimes_R n}
= R\lala\kk\Q_1+\I \rara \]
holds.

Now, we provide some examples for weight quivers and modulations.
Each quiver $\Q=(\Q_0,\Q_1,\fcts,\fctt)$ can be seen as a {\defines trivial} weight quiver $(\Q,\vecd)$ with $\vecd=(1,\cdots,1)$,
and then the path algebra $\kk\Q$ is isomorphic to $R\la\kk\Q_1\ra$ which provides a $\kk$-modulation
$((\kk\e_v)_{v\in \Q_0}, (\kk\alpha)_{\alpha\in \Q_1})$ of $\Q$. Here,
$R=\prod_{i\in\Q_0} \kk\e_{i}$; and, for each arrow $\alpha\in \Q_1$,
it is clear that $A_{\alpha} = \kk\alpha$ is a $(\kk\e_{\fcts(\alpha)}, \kk\e_{\fctt(\alpha)})$-bimodule.

\begin{example} \rm \label{exp:modulation w=0}
For example, let $\alg=\kk\Q$ be a $\kk$-algebra over an algebraically closed field $\kk$
given by the quiver $\Q = \xymatrix{1 \ar[r]^a & 2 \ar[r]^b & 3}$,
and $R$ be the ring $\kk\Q_0$ which is isomorphic to a semi-simple algebra $\kk^{\times \Q_0}$.
Then $(\kk\Q_1)^{\otimes_R \>=3} =0$, and so we obtain
\begin{align*}
  \alg & = (\kk\e_1+\kk\e_2+\kk\e_3) \oplus (\kk a+\kk b) \oplus \kk a b\\
& \cong  R\la\kk\Q_1\ra = R \oplus \kk\Q_1 \oplus (\kk\Q_1)^{\otimes_R 2},
\end{align*}
where $R = \kk\e_1+\kk\e_2+\kk\e_3$, $A=\kk\Q_1=\kk a+\kk b$, and $A^{\otimes_R 2} = (\kk a+\kk b) \otimes_R (\kk a+\kk b) = \kk a\otimes b \cong \kk ab$ $=\kk\Q_2 \cong \kk\Q_1\otimes_R\kk\Q_1$.
All $\kk$-vector spaces $\kk\e_1$, $\kk\e_2$, $\kk\e_3$ are division $\kk$-algebra.
The $\kk$-vector space $\kk a$ is a $(\kk\e_1,\kk\e_2)$-bimodule whose left $\kk\e_1$-action
is given by $(k_{\e_1}\e_1, k_aa)\mapsto k_{\e_1}k_a \e_1\otimes a =  k_{\e_1}k_aa$
since the tensor $\e_1\otimes a$ is defined as the multiplication of paths $\e_1\otimes a := \e_1a=a$,
and whose right $\kk\e_2$-action is induced by $a\e_2=a$ by a dual way.
Similarly, $\kk b$ is a $(\kk\e_2,\kk\e_3)$-bimodule, and $\kk ab \cong \kk(a\otimes b)$ is a $(\kk\e_1,\kk\e_3)$-bimodule.
Thus, we obtain a $\kk$-modulation $((\kk\e_1, \kk\e_2, \kk\e_3), (\kk a, \kk b))$ of $\Q$
which can be written as
\[ \xymatrix{D_1 \ar[r]^{A_a} & D_2 \ar[r]^{A_b} & D_3} =
\xymatrix{\kk\e_1 \ar[r]^{\kk a} & \kk\e_2 \ar[r]^{\kk b} & \kk e_3}.  \]
The above $\kk$-modulation describes $\kk\Q$.
\end{example}

Next, we provide an example for a modulation of a non-trivial weight quiver.
For any weight quiver $(\Q,\vecd)$ with $\vecd=(d_i)_{i\in\Q_0}$,
let $\FF$ be a base field, $\EE$ is an extension field of $\FF$ with $[\EE:\FF]=d:=\mathrm{lcm}(d_i\mid i\in \Q_0)$,
and, for any $i\in\Q_0$, $\FF_i$ be an extension field of $\FF$ such that $\FF\subseteq\FF_i\subseteq \EE$ and $[\FF_i:\FF]=d_i$ hold.
Then for any element $\vecg=(g_{\alpha})_{\alpha\in\Q_1}$ in the Cartesian product
$\prod_{\alpha\in\Q_1} \Gal(\FF_{\fcts(\alpha)} \cap \FF_{\fctt(\alpha)} / \FF)$
of Galois groups, define
\begin{center}
   $\displaystyle R = \prod_{i\in\Q_0} \FF_i\e_i$,
   $\displaystyle A_{\alpha} = \FF_{\fcts(\alpha)}\otimes_{\FF_{\fcts(\alpha)}\cap\FF_{\fctt(\alpha)}}
  \FF_{\fctt(\alpha)}^{g_{\alpha}}$,
and $\displaystyle A = \bigoplus_{\alpha\in \Q_1} A_{\alpha}$,
\end{center}
where $\FF_{\fctt(\alpha)}^{g_{\alpha}}$ is the field $\FF_{\fctt(\alpha)}$ with
the right $\FF_{\fctt(\alpha)}$-action
\[ \FF_{\fctt(\alpha)}^{g_{\alpha}} \times \FF_{\fctt(\alpha)}
   \to \FF_{\fctt(\alpha)}^{g_{\alpha}},
   (x,z) \mapsto xz \]
given by the multiplication in field $\FF_{\fctt(\alpha)}$
and the left $\FF_{\fcts(\alpha)}\cap \FF_{\fctt(\alpha)}$-action
\[(\FF_{\fcts(\alpha)}\cap \FF_{\fctt(\alpha)}) \times \FF_{\fctt(\alpha)}^{g_{\alpha}}
   \to \FF_{\fctt(\alpha)}^{g_{\alpha}},
   (z,x) \mapsto g_a(z)x \]
given by $g_a \in \Gal(\FF_{\fcts(\alpha)}\cap\FF_{\fctt(\alpha)}/\FF)$. Then
\begin{center}
  $((\FF_i)_{i\in\Q_0}, (A_{\alpha})_{\alpha\in\Q_1})$,
   written as
 $( \xymatrix{ \FF_{\fcts(\alpha)}
     \ar[r]^{A_{\alpha}}
   & \FF_{\fctt(\alpha)} }
   )_{\alpha\in\Q_1}$ for clarity,
\end{center}
is an $\FF$-modulation of $(\Q,\vecd)$ corresponded by the tensor ring
\begin{center}
  $\alg(\Q,\vecd,\vecg, \EE/\FF, (\FF_i/\FF)_{i\in\Q_0})) := R\la A\ra$.
\end{center}
We call $\vecg$ as above a {\defines modulation function}.
In particular, for any two arrow $a$ and $b$ with $\fctt(a)=\fcts(b)$ and any
$\lambda\in\bigcap_i \FF_i \subseteq \EE$, we have
$a \otimes b \in A_a \otimes_{\FF_{\fctt(a)}\cap\FF_{\fcts(b)}} A_b$
and, by \ref{(2.3)}, have
\[ (a\otimes b)\lambda = a\otimes (b\lambda) = a\otimes (g_b(\lambda) b)
= (ag_b(\lambda))\otimes b = (g_a(g_b(\lambda))a)\otimes b. \]
It follows that if $\lambda \in \FF \subseteq \bigcap_i \FF_i$, then
\[ (a\otimes b)\lambda = a \otimes (b\lambda) = a\otimes (\lambda b)
= (a\lambda) \otimes b = (\lambda a) \otimes b) = \lambda(a\otimes b). \]
Thus, the tensor ring $R\la A\ra$ is an $\FF$-algebra.

\begin{example} \rm \label{exp:modulation w ne 0}
Take $(\Q,\vecd)$ is a weight quiver given by $\Q = $
\begin{figure}[H]
\centering
\begin{tikzpicture}[scale = 0.75]
\draw (-1.73,-1) node{2} (0,2) node{1} (1.73,-1) node{3};
\draw [rotate=  0][->] (-1.43,-1) -- (1.43,-1);
\draw [rotate=120][->] (-1.43,-1) -- (1.43,-1);
\draw [rotate=240][->] (-1.43,-1) -- (1.43,-1);
\draw [rotate=30+120] (1.3,0) node{$a$};
\draw [rotate=30+240] (1.3,0) node{$b$};
\draw [rotate=30+360] (1.3,0) node{$c$};
\end{tikzpicture}
\end{figure}
\noindent
and $\vecd=(2,2,1)$,
and $\EE$ and $\FF$ are two fields with $[\EE:\FF]=2$.
Let $\FF_1=\EE$, $\FF_2=\EE$, and $\FF_3=\FF$, then it is clearly that $\FF_1$, $\FF_2$ and $\FF_3$ are three finite-dimensional division $\FF$-algebras corresponded by the vertices $1$, $2$, and $3$, respectively.
For an arbitrary modulation function
\begin{center}
  $\vecg = \displaystyle (g_{\alpha})_{\alpha\in\Q_1} \in
  \prod_{\alpha\in\Q_1} \Gal(\FF_{\fcts(\alpha)}\cap\FF_{\fctt(\alpha)}/\FF)$,
\end{center}
we define
\begin{itemize}
  \item $R = \FF_1 \times \FF_2 \times \FF_3$;
  \item $A=A_a\oplus A_b\oplus A_c$ is an $(R,R)$-bimodule, where:
\begin{itemize}
  \item $A_a := \FF_1 \otimes_{\FF_1\cap\FF_2} \FF_2^{g_a}$ is a $(\FF_1,\FF_2)$-bimodule,
  \item $A_b := \FF_2 \otimes_{\FF_2\cap\FF_3} \FF_3^{g_b}$ is a $(\FF_2,\FF_3)$-bimodule,
  \item $A_c := \FF_3 \otimes_{\FF_3\cap\FF_1} \FF_1^{g_c}$ is a $(\FF_3,\FF_1)$-bimodule.
\end{itemize}
\end{itemize}
Then
\[ \alg(\Q,\vecd,\vecg,\EE,(\FF_i)_{i\in\Q_0})
 := R \oplus (A_a\oplus A_b\oplus A_c) \oplus (A_{ab}\oplus A_{bc}\oplus A_{ca}) \oplus \cdots \]
is a tensor ring. Here, $A_{ab} = A_a \otimes_{\FF_2} A_b$ is a $(\FF_1, \FF_3)$-bimodule by the following fact
\begin{align*}
    A_{ab}
& = A_a \otimes_{\FF_2} A_b \\
& = (\FF_1 \otimes_{\FF_1\cap\FF_2} (\FF_2^{g_a})_{\FF_2})
    \otimes_{\FF_2}
    (\FF_2 \otimes_{\FF_2\cap\FF_3} (\FF_3^{g_b})_{\FF_3}) \\
& = {_{\FF_1}}(\FF_1 \otimes_{\FF_1\cap\FF_2} \FF_2^{g_a} \otimes_{\FF_2\cap\FF_3} \FF_3^{g_b})_{\FF_3}.
\end{align*}
Simlarly, one can check that $A_{ab}$, $A_{bc}$, $A_{abc}$, $\ldots$ are bimodules,
and we can obtain an $\FF$-modulation of $(\Q,\vecd)$ by $\vecg=(g_a,g_b,g_c) \in \Gal(\FF_1\cap\FF_2/\FF) \times \Gal(\FF_2\cap\FF_3/\FF) \times \Gal(\FF_3\cap\FF_1/\FF)$ as follows.
\begin{figure}[H]
\centering
\begin{tikzpicture} 
\draw (-1.73,-1) node{$\FF_2$} (0,2) node{$\FF_1$} (1.73,-1) node{$\FF_3$};
\draw [rotate=  0][->] (-1.43,-1) -- (1.43,-1);
\draw [rotate=120][->] (-1.43,-1) -- (1.43,-1);
\draw [rotate=240][->] (-1.43,-1) -- (1.43,-1);
\draw [rotate= 90](2.5,0) node{$d_1=2$};
\draw [rotate=210](2.5,0) node{\rotatebox{120}{$d_2=2$}};
\draw [rotate=330](2.5,0) node{\rotatebox{240}{$d_3=1$}};
\draw [rotate=30+120] (0.7,0) node{\rotatebox{ 60}{$A_a$}};
\draw [rotate=30+120] (1.3,0) node{\rotatebox{ 60}{$\FF_1 \otimes_{\FF_1\cap\FF_2} \FF_2^{g_a}$}};
\draw [rotate=30+  0] (0.7,0) node{\rotatebox{-60}{$A_c$}};
\draw [rotate=30+  0] (1.3,0) node{\rotatebox{-60}{$\FF_3 \otimes_{\FF_3\cap\FF_1} \FF_1^{g_c}$}};
\draw [rotate=30+240] (0.7,0) node{\rotatebox{  0}{$A_b$}};
\draw [rotate=30+240] (1.3,0) node{\rotatebox{  0}{$\FF_2 \otimes_{\FF_2\cap\FF_3} \FF_3^{g_b}$}};
\end{tikzpicture}
\end{figure}
\noindent
\end{example}

\subsection{Representations of weight quivers}

Let $\alg=\alg(\Q,\vecd,\vecg,\EE,(\FF_i)_{i\in \Q_0})$ and $\alg'=\alg(\Q',\vecd',\vecg',\EE',(\FF_i')_{i\in \Q_0})$ be two tensor rings. Then an {\defines algebraical homomorphism} (=homomorphism for simplicity)
between $\alg$ and $\alg'$ is a homomorphism
\[ h: \alg \to \alg' \]
of Abel groups such that
\begin{itemize}
  \item $h(a\otimes b) = h(a)\otimes h(b)$ holds for all arrows $a,b\in \Q_1$.
  \item $h(\lambda a) = \lambda h(a)$ holds for all $a\in\Q_1$ and $\lambda\in\FF$.
\end{itemize}
Then for any finite-dimensional $\FF$-vector space $V$ with $\dim_{\FF}V = n$, its endomorphism $\End_{\FF}V \cong \Mat_{n\times n}(\FF)$ is a ring which can be seen as a tensor ring
\[ \End_{\FF}V \cong R\la A \ra,  \]
where
\begin{itemize}
  \item $R = \displaystyle \prod_{1\=<i\=<n}\FF\bfE_{ii}$ (for each $1\=<i,j\=<$, $\bfE_{ij}$ is the $n\times n$ matrix whose element in the $i$-th row and $j$-th column is $1$, and the other elements are $0$);
  \item $A = \displaystyle \bigoplus_{1< i \=<n} \bfE_{i,i-1} \oplus \bigoplus_{1\=< j <n} \bfE_{j,j+1}$,
  \item and $\bfE_{ij}\otimes \bfE_{i'j'} := \bfE_{ij}\bfE_{i'j'}$.
\end{itemize}
Thus, for any $\alg=\alg(\Q,\vecd,\vecg,\EE,(\FF_i)_{i\in \Q_0})$, each algebraical homomorphism
\begin{center}
  $h: \alg \to (\End_{\FF}V)^{\op}, ~ r\mapsto h_r$
\end{center}
induces a right $\alg$-action
\begin{center}
  $V\times \alg \to V, ~ (v,r)\mapsto v.r := h_r(v)$
\end{center}
such that the following five facts
hold for all $v, v_1, v_2\in V$, $r, r_1,r_2\in \alg$, and $\lambda\in\FF$:
\begin{enumerate}[label=(M\arabic*)]
  \item 
    $v.(r_1+r_2) = v.r_1+v.r_2$;
    \label{M1}
  \item 
    $(v_1+v_2).r = v_1.r+v_2.r$; \label{M2}
  \item 
    $m.(r_1r_2) = (m.r_1).r_2$; \label{M3}
  \item
    $m.1_{\alg}=m$ ($1_{\alg}$ is the identity of $\alg$); \label{M4}
  \item
    $m.(r\lambda) = (m.r)\lambda = (m\lambda).r$. \label{M5}
\end{enumerate}
Dually, each algebraic homomorphism
\begin{center}
  $h: \alg \to \End_{\FF}V, ~ r\mapsto h_r$
\end{center}
induces a left $\alg$-action
\begin{center}
  $\alg\times V \to V, (v,r)\mapsto r.v := h_r(v)$
\end{center}
such that the following five facts
hold for all $v, v_1, v_2\in V$, $r, r_1,r_2\in \alg$, and $\lambda\in\FF$:
\begin{enumerate}[label=(\arabic*M)]
  \item
    $(r_1+r_2).v = r_1.v+r_2.v$;
    \label{1M}
  \item
    $r.(v_1+v_2) = r.v_1+r.v_2$; \label{2M}
  \item
    $(r_1r_2).m = r_1.(r_2.m)$; \label{3M}
  \item
    $1_{\alg}.m=m$; \label{4M}
  \item
    $(\lambda r).m = \lambda(r.m) = r.(\lambda m)$. \label{5M}
\end{enumerate}

\begin{definition} \rm
Let $\alg = \alg(\Q,\vecd,\vecg,\EE,(\FF_i)_{i\in \Q_0})$.
\begin{itemize}
  \item[(1)]
    A {\defines right $\alg$-module} (or {\defines right $\alg$-representation}) is an $\FF$-vector space $V$
with a right $\alg$-action $V \times \alg \to V$ such that the conditions \ref{M1}--\ref{M5} hold.
  \item[(2)]
    A {\defines left $\alg$-module} (or {\defines left $\alg$-representation}) is an $\FF$-vector space $V$
with a left $\alg$-action $\alg\times V \to V$ such that the conditions \ref{1M}--\ref{5M} hold.
\end{itemize}
\end{definition}

Each right $\alg$-module $M=M_{\alg}$ has a decomposition
\[ M = M 1_{\alg} = M \sum_{i\in\Q_0} \e_i \cong \bigoplus_{i\in\Q_0} M\e_i \]
such that for any path $\wp = a_1\cdots a_l$, we have
\begin{align*}
 M \e_{\fcts(a_1)} \cdot \wp
& = M \e_{\fcts(a_1)}
  \otimes_{\FF_{\fcts(a_1)}} A_{a_1}
  \otimes_{\FF_{\fctt(a_1)}\cap\FF_{\fcts(a_2)}} A_{a_2}
  \otimes_{\FF_{\fctt(a_2)}\cap\FF_{\fcts(a_3)}} \cdots
  \otimes_{\FF_{\fctt(a_{l-1})}\cap\FF_{\fcts(a_l)}} A_{a_l} \\
& = M \e_{\fcts(a_1)} \otimes_{\FF_{\fcts(a_1)}}
  \bigg(
    \bigotimes_{i=1}^l \FF_{\fcts(a_i)} \otimes_{\FF_{\fcts(a_i)} \cap \FF_{\fctt(a_i)}} \FF_{\fctt(a_i)}^{g_{a_i}}
  \bigg) \otimes_{\FF_{\fctt(a_l)}} \FF_{\fctt(a_l)}\e_{\fctt(a_l)} \\
& \subseteq  M \e_{\fctt(a_l)}.
\end{align*}
It follows that each $M$ can be corresponded to a sequence
\begin{align}\label{repr 1}
  (M\e_i, \varphi_{\alpha})_{i\in\Q_0, \alpha \in \Q_1}
\end{align}
given by $\FF_i$-vector spaces $(M\e_i)_{i\in\Q_0}$ and $\FF$-linear maps
$(\varphi_{\alpha}: M\e_{\fcts(\alpha)} \otimes_{\FF_{\fcts(\alpha)}}
  A_{\alpha} \to M\e_{\fctt(\alpha)})_{\alpha\in\Q_1}$.
Conversely, for right $\FF_i$-modules $(M_i)_{i\in\Q_0}$ and $\FF$-linear maps
$(M_{\alpha}: M_{\fcts(\alpha)} \otimes_{\FF_{\fcts(\alpha)}} A_{\alpha} \to M_{\fctt(\alpha)})_{\alpha\in\Q_1}$, we obtain a sequences
\begin{align}\label{repr 2}
  (M_i, M_{\alpha})_{i\in\Q_0, \alpha \in \Q_1}
\end{align}
which induces a right $\alg$-module $\displaystyle M := \bigoplus_{i\in \Q_0} M_i$
with the right $\alg$-action $M\times \alg \to M$ sending each
$(m_{\fcts(\alpha)}, a)$ in $ M_{\fcts(\alpha)} \times A_{\alpha}$ $(\subseteq M\times \alg)$
to the element $m_{\fcts(\alpha)} \otimes a$ in the tensor $M_{\fcts(\alpha)} \otimes_{\FF_{\fcts(\alpha)}} A_{\alpha}$
($\subseteq M_{\fctt(\alpha)} \subseteq M$). The sequence given in (\ref{repr 1}) or (\ref{repr 2}) is called
a {\defines right quiver representation} of $(\Q,\vecd,\vecg,\EE,(\FF_i)_{i\in \Q_0})$.

Now, let $(\Q,\vecd,\vecg,\EE,(\FF_i)_{i\in \Q_0})_{\repcat}$ be the category whose objects are right quiver representations of $\alg$
and, for any objects $M=(M_i, M_{\alpha})_{i\in\Q_0, \alpha \in \Q_1} $
and $N = (N_i, N_{\alpha})_{i\in\Q_0, \alpha \in \Q_1}$,
each morphism $h: M \to N$ is a family of $\FF$-linear maps $(h_i:M_i\to N_i)_{i\in\Q_0}$ such that
each $h_i$ is a right $\FF_i$-homomorphism and the diagram
\[
\xymatrix{
   M_i \otimes_{\FF_i} A_{\alpha}
   \ar[rr]^{M_{\alpha}}
   \ar[d]^{h_i\otimes 1_{A_{\alpha}}}
&& M_j
   \ar[d]^{h_j}
\\
   N_i \otimes_{\FF_i} A_{\alpha}
   \ar[rr]^{N_{\alpha}}
&& M_j
}
\]
commutes, where $\alpha$ is an arbitrary arrow from $i$ to $j$. Then $(\Q,\vecd,\vecg,\EE,(\FF_i)_{i\in \Q_0})_{\repcat}$ describes the finite-dimensional right $\alg$-module category $\modcat_{\alg}$.

Dually, each left $\alg$-module $M={_\alg}M$ has a decomposition
$\displaystyle M = \bigoplus_{i\in\Q_0} \e_i M$ such that for any $\wp = a_1\cdots a_l$,
we have $\wp\e_{\fctt(a_l)}M \subseteq \e_{\fcts(a_1)}M$.
It follows that each $M$ can be It follows that each $M$ can be corresponded by a sequence
$(\e_i M, \varphi_{\alpha})_{i\in\Q_0, \alpha \in \Q_1}$ given by $\FF_i$-vector spaces $(\e_iM)_{i\in\Q_0}$
and $\FF$-linear maps $(\varphi_{\alpha}: A_{\alpha}
\otimes_{\FF_{\fctt(\alpha)}}\e_{\fctt(\alpha)} M
\to \e_{\fcts(\alpha)}M)_{\alpha\in\Q_1}$.
Conversely, for left $\FF_i$-modules $(M_i)_{i\in\Q_0}$ and $\FF$-linear maps
$(M_{\alpha}:  A_{\alpha} \otimes_{\FF_{\fctt(\alpha)}} M_{\fctt(\alpha)}
 \to M_{\fcts(\alpha)})_{\alpha\in\Q_1}$, we obtain a sequences
$(M_i, M_{\alpha})_{i\in\Q_0, \alpha \in \Q_1}$
which induces a left $\alg$-module $\displaystyle M := \bigoplus_{i\in \Q_0} M_i$ by a dual way.
The sequence $(\e_i M, \varphi_{\alpha})_{i\in\Q_0, \alpha \in \Q_1}$ or
$(M_i, M_{\alpha})_{i\in\Q_0, \alpha \in \Q_1}$ are called a {\defines left quiver representation} of $(\Q,\vecd,\vecg,\EE,(\FF_i)_{i\in \Q_0})$.
Furthermore, let $(\Q,\vecd,\vecg,\EE,(\FF_i)_{i\in \Q_0})_{\Repcat}$ be the category whose objects are left quiver representations of $\alg$
and, for any objects $M=(M_i, M_{\alpha})_{i\in\Q_0, \alpha \in \Q_1} $
and $N = (N_i, N_{\alpha})_{i\in\Q_0, \alpha \in \Q_1}$,
each morphism $h: M \to N$ is a family of $\FF$-linear maps $(h_i:M_i\to N_i)_{i\in\Q_0}$ such that
each $h_i$ is a left $\FF_i$-homomorphism and the diagram
\[
\xymatrix{
  A_{\alpha} \otimes_{\FF_i}  M_i
   \ar[rr]^{M_{\alpha}}
   \ar[d]^{1_{A_{\alpha}} \otimes h_i }
&& M_j
   \ar[d]^{h_j}
\\
  A_{\alpha} \otimes_{\FF_i}  N_i
   \ar[rr]^{N_{\alpha}}
&& M_j
}
\]
commutes, where $\alpha$ is an arbitrary arrow from $j$ to $i$.
Then ${_{\repcat}}\alg$ describes ${_{\alg}}\modcat$ and $\alg_{\repcat}$ describes $\modcat_{\alg}$. To be more precise, we have the following theorem.

\begin{theorem}
Let $\alg$ be a tensor ring $\alg(\Q,\vecd,\vecg,\EE,(\FF_i)_{i\in \Q_0})$ of a weight quiver $(\Q,\vecd)$.
Then there exists an $\FF$-equivalence of categories
\[ \Modcat_{\alg} \mathop{\longrightarrow}\limits^{\simeq} (\Q,\vecd,\vecg,\EE,(\FF_i)_{i\in \Q_0})_{\Repcat} \]
which sends each right $\alg$-module $M$ to the quiver representation $(M_i, \varphi_{\alpha})_{i\in \Q_0, \alpha \in \Q_1}$ decided by $\varphi_{\alpha}:M\e_i \otimes_{\FF_i} A_{\alpha} \to M\e_j$.
One can obtain a dual result
\[ {_\alg}\Modcat \mathop{\longrightarrow}\limits^{\simeq} {_{\Repcat}}(\Q,\vecd,\vecg,\EE,(\FF_i)_{i\in \Q_0}) \]
which describe the finite-dimensional left $\alg$-module category ${_\alg}\modcat$ by a similar way.
\end{theorem}

\section{Norms} \label{sect:norm}

An algebra with a norm is called a normed algebra. Furthermore, if a normed algebra is complete, then is called a Banach algebra. In this section, we will consider normed tensor rings and normed module over normed tensor rings.

\subsection{Normed tensor rings} \label{sect:norm tens ring}

Assume $\FF$ is a field with a norm $|\cdot|: \FF \to \RR^{\>=0}$.
Let $\alg = \alg(\Q,\vecd,\vecg, \EE, (\FF_i)_{i\in\Q_0}))$ be a tensor ring,
and, as an algebra over $\FF$, we put that its dimension $\dim_{\FF}\alg$ is finite. Then
\[ \alg = \spa_{\FF}(\basis{\alg}) = \bigboxplus_{i=1}^{n} \FF e_i
~ \bigg(\cong \bigoplus_{i=1}^n \FF e_i \bigg) \]
($\basis{\alg} = \{e_i \mid 1\=< i\=< n=\dim_{\FF}\alg\}$ is a basis of $\alg$),
which admits that each element $a\in\alg$ is of the form
\[ a = \sum_{i=1}^{n} f_i e_i, ~ f_i\in\FF. \]
Thus, for any $1\=< p\in\RR^+$ and map $\norm:\basis{\alg}\to\RR^{\>=0}$, the formula
\begin{align}\label{formula:norm on alg}
  \Vert a\Vert_p := \bigg(\sum_{i=1}^n |f_i|^p \norm(e_i)^p\bigg)^{\frac{1}{p}}
\end{align}
admits a finite-dimensional norm $\FF$-vector space $(\alg,\norm,\Vert\cdot\Vert_p)$,
see \cite[Proposition 3.1]{LLHZ2025}.

\begin{definition} \rm
A {\defines normed tensor ring} is a triple $(\alg,\norm,\Vert\cdot\Vert_p)$ (=$\alg$ for short),
where $\norm:\basis{\alg}\to\RR^{\>=0}$ and $\Vert\cdot\Vert_p:\alg\to\RR^{\>=0}$
are called the {\defines normed basis function} and {\defines norm} of $\alg$, respectively.
\end{definition}

Let $A=\alg(\Q_A,\vecd_A,\vecg_A,\FF,(\FF_i)_{i\in(\Q_A)_0})$ and
$B=\alg(\Q_B,\vecd_B,\vecg_B,\FF,(\FF_i)_{i\in(\Q_B)_0})$ respectively be
two tensor rings of weight quivers $\Q_A$ and $\Q_B$
with $\dim_{\FF}A<\infty$ and $\dim_{\FF}B<\infty$,
and $\homo: A \to B$ be a homomorphism of two tensor rings.
Consider the basis of $B$ given by the modulation
\[ (\xymatrix{
\FF_{\fcts(\beta)}
\ar[r]^{B_{\alpha}}
& \FF_{\fctt(\beta)}
})_{\beta\in(\Q_B)_1} \]
of $\Q_B$, where each $B_{\beta} = \FF_{\fcts(\beta)}
       \otimes_{\FF_{\fcts(\beta)}\cap\FF_{\fctt(\beta)}}
    \FF_{\fctt(\beta)}^{g_{B,\beta}}$,
as an $\FF$-vector space, has a finite dimension $\dim_{\FF} B_{\beta} = d_{\beta} < \infty$. Then
\[ B \cong R_B\la \FF(\Q_B)_1\ra = \prod_{n=1}^{L_B} (\FF(\Q_B)_1)^{\otimes_R n} \]
holds for some $L_B \in \NN$, and, in particular, we have
\begin{center}
$\displaystyle \dim_{\FF} B = \sum_{n\>=0} \dim_{\FF}(\FF(\Q_B)_1)^{\otimes_R n}
= \sum_{i\in(\Q_B)_0} d_i
 + \sum_{l=1}^{L_B}
   \sum_{
     \wp=\beta_1\cdots\beta_l\in(\Q_B)_l
     \atop
     (\wp\ne 0)
     }
     d_{\beta_1}\cdots d_{\beta_l}
$
\end{center}
if $\FF_i\cap \FF_j =\FF$ holds for all $i \ne j\in(\Q_B)_0$.
Thus, $B$ has a basis $\basis{B} = \{e_{B,i} \mid 1\=< i \=< \dimB = \dim_{\FF} B \}$,
and for any $1\=<p \in\RR^+$ and $\norm_B: \basis{B} \to \RR^{\>=0}$,
the formula (\ref{formula:norm on alg}) induces the following map
\begin{align}\label{norm B}
  \Vert\cdot\Vert_{B,p}: B \to \RR^{\>=0}, ~
  b=\sum_{i=1}^{\dimB} f_ie_{B,i}
  \mapsto
   \bigg(
     \sum_{i=1}^{\dimB} |f_i|^p\norm(e_{B,i} )^p
   \bigg)^{\frac{1}{p}},
\end{align}
which defines a norm of $B$.

\def\bfvert{\pmb{\vert}}

For a basis $\basis{B}$ of $B$, we know that a map $\norm_B:\basis{B}\to \RR^{\>=0}$
provides a norm $\Vert\cdot\Vert_{B,p}$ defined on $B$ by using (\ref{formula:norm on alg}).
Then, for any homomorphism $\homo: A\to B$, $\homo$ induces a map
\[ \bfvert\cdot\bfvert := \bfvert\cdot\bfvert_{\homo} : A \to \RR^{\>=0}, ~ a \mapsto \Vert \homo(a)\Vert_{B,p} \]
satisfying the following three facts:
\begin{enumerate}[label=(\arabic*)]
  \item $\bfvert a\bfvert \>=0$;
  \item $\bfvert \lambda a\bfvert = \Vert \homo(\lambda a)\Vert_{B,p} = |\lambda |\Vert \homo(a)\Vert_{B,p} = |\lambda |\bfvert a\bfvert$,
    ($\forall \lambda \in\FF, a\in A$);
  \item $\bfvert a_1+a_2\bfvert = \Vert \homo(a_1)+\homo(a_2)\Vert_{B,p}
    \=< \Vert \homo( a_1)\Vert_{B,p}+\Vert \homo(a_2)\Vert_{B,p}
     = \bfvert a_1\bfvert + \bfvert a_2 \bfvert$ ($\forall a_1,a_2\in A$).
\end{enumerate}
Thus, $\homo$ induces a seminorm $\bfvert \cdot\bfvert$ defined on $A$. Here, $A$ is seen as an $\FF$-vector space.
It is easy to prove that $\bfvert a \bfvert = 0$ if and only if $a \in \Ker(\homo)$.
Then $\bfvert\cdot\bfvert$ induces a map $A/\Ker(\homo) \to \RR^{\>=0}$, $a+\Ker(\homo) \mapsto \bfvert a \bfvert$
which is well-defined since $\bfvert a \bfvert$ $-$ $\bfvert k \bfvert$ $=$ $\bfvert a \bfvert$ $\=<$ $\bfvert a + k \bfvert$
$\=<$ $\bfvert a \bfvert + \bfvert k \bfvert$ $= \bfvert a \bfvert$ admits that $\bfvert a +k \bfvert = \bfvert a \bfvert$ holds for all $k\in\Ker(\homo)$.

\subsection{Normed representations}

Let $\tau$ be a homomorphism of $\FF$-algebras $\tau: A \to \FF$
and $|\cdot|: \FF \to \RR^{\>=0}$ be a norm defined on $\FF$.
In \cite[Definition 4.1]{LLHZ2025}, a $\tau$-normed right $A$-module
over a finite-dimensional $\FF$-algebra $A$ is a $\FF$-vector space $M$
with two maps $\Vert\cdot\Vert_M: M \to \RR^{\>=0}$ and $h: A \to \End_A(M)$ such that
\[ \Vert ma \Vert_M = \Vert m \Vert_M |\tau(a)| \]
and
\[ h(a_1a_2) = h(a_1)h(a_2) \]
hold for all $m\in M$ and $a, a_1, a_2\in A$.
Therefore, each $\tau$-normed $A$-module is triple $(M,h,\Vert\cdot\Vert_M)$
of an $\FF$-vector space $M$, a homomorphism $h: A \to \End_A(M)$ of $\FF$-algebras,
and a norm $\Vert\cdot\Vert_M: M \to \RR^{\>=0}$.
One can define $\tau$-normed left $A$-module in a dual way.

Next, we provide a more general definition of a normed module.

\begin{definition}[Normed module] \rm
Let $M={_AM_B}$ be an $(A,B)$-bimodule whose left $A$-action and right $B$-action respectively are
$A\times M \to M$, $(a,m)\mapsto a.m$ and $M\times B \to M$, $(m,b)\mapsto m.b := mb$ such that
$(a.m).b=a.(m.b)$ holds for all $a\in A$, $b\in B$, and $m\in M$.
A {\defines $\homo$-norm} $\Vert\cdot\Vert_M$ defined on $M$ is a map
\[ \Vert\cdot\Vert_M: M \to \RR^{\>=0} \]
such that:
\begin{enumerate}[label=(N\arabic*)]
  \item
    $\Vert\cdot\Vert_M$ is a norm defined on $\FF$-vector space $M=M_{\FF}$.
    \label{N1}
  \item
    $\Vert a.m.b \Vert_M = \bfvert a \bfvert \Vert m\Vert_M \Vert b\Vert_{B,p}$
    $(= \Vert \homo(a)\Vert_{B,p} \Vert m\Vert_M \Vert b\Vert_{B,p})$ .
    \label{N2}
\end{enumerate}
\end{definition}

\section{Two categories} \label{sect:two cats}

Let $\FF$ be a completed field. Keep the notations from Subsection \ref{sect:norm tens ring},
$A=\alg(\Q_A,\vecd_A,\vecg_A,\FF$, $(\FF_i)_{i\in(\Q_A)_0})$ and
$B=\alg(\Q_B,$ $\vecd_B,\vecg_B,\FF,(\FF_i)_{i\in(\Q_B)_0})$ are tensor rings whose dimensions
$\dimA = \dim_{\FF} A$ and $\dimB = \dim_{\FF} B$ are finite,
and $\homo:A\to B$ is a homomorphism. Then there exists a basis $\basis{A}=\{e_{A,i} \mid 1\=< i\=< \dimA\}$ of $A$
such that $\displaystyle A = \sum_{i=1}^{\dimA} \FF e_{A,i}$ holds.
We assume that $\FF$ contains totally ordered subset a $\II=(\II,\preceq)$ in this paper,
then $\II$ can be written as $[c,d]_{\FF} := \{\lambda\in\FF \mid c\preceq \lambda\preceq d\}$,
where $c$ and $d$ are minimal and maximal in $\II$, respectively.
If $c=d$, then $[c,d]_{\FF} = \{c\} = \{d\}$.
Let $\mu_{\FF}$ be an arbitrary measure defined on $\FF$,
then for the totally ordered subset $\II$, $\mu_{\FF}$ induces a measure $\mu_{\II_A}$ defined on
\[ \II_A = [c,d]_A := \sum_{i=1}^{\dimA} [c,d]_A e_{A,i}
~\mathop{\simeq}\limits^{1-1}~
\II^{\times \dimA} := \prod_{i=1}^{\dimA} \II \]
such that $\mu_{\II_A}(\II_A) = \mu_{\FF}([c,d])^{\dimA}$.
In this section, we introduce two important categories of this paper
by given $1 \=< p \in \RR$, $A$, $B$, $\homo:A\to B$, and $\mu_{\II_A}$.

\subsection{\texorpdfstring{Categories $\scrN_{\homo}^p$ and $\scrA_{\homo}^p$}{Categories Np and Ap}}\label{subsect:two cats}

\subsubsection{Normed module categories} \label{subsubsect:Np}

The following lemma shows that $2^{\dimA}$ $\homo$-normed $(A,B)$-bimodules is also a $\homo$-normed $(A,B)$-bimodule.
This fact will be used to define the category $\scrN_{\homo}^p$.

\begin{lemma} \label{lemm:dir sum}
Let $X$ be the direct sum $\displaystyle X := \bigoplus_{i=1}^{2^{\dimA}} X_i$
of $2^{\dimA}$ $\homo$-normed $(A,B)$-bimodules $X_1$, $X_2$, $\ldots$, $X_{2^{\dimA}}$.
For any disjoint union $\II_A = \bigcup\limits_{i=1}^{\dimA} \II_i$,
The $(A,B)$-bimodule $X$ equipped with the map
\[\Vert \cdot \Vert_X: X \to \RR^{\>=0},
  (x_1,\ldots, x_{2^{\dimA}})
 \mapsto
 \bigg(
   \sum_{i=1}^{2^{\dimA}}
   \bigg(
     \frac{\mu_{\II_A}(\II_i)}{\mu_{\II_A}(\II_A)}
   \bigg)^p
   \Vert x_i \Vert_{X_i}^p
 \bigg)^{\frac{1}{p}} \]
is a $\homo$-normed $(A,B)$-bimodule.
\end{lemma}

\begin{proof}
First of all, for each summand $\bigg(
     \frac{\mu_{\II_A}(\II_i)}{\mu_{\II_A}(\II_A)}
   \bigg)^p
   \Vert x_i \Vert_{X_i}^p$, let $\frac{\mu_{\II_A}(\II_i)}{\mu_{\II_A}(\II_A)} x_i = \tilde{x}_i$,
then this summand is of the form $\Vert \tilde{x_i} \Vert_{X_i}^p$.
Therefore, we can assume $\frac{\mu_{\II_A}(\II_i)}{\mu_{\II_A}(\II_A)} = c$ holds for all $i$ such that the sum \[\sum_{i=1}^{\dimA} \frac{\mu_{\II_A}(\II_i)}{\mu_{\II_A}(\II_A)}  = c\dimA=1 \] in this proof.
Second, $\Vert \cdot \Vert_X$ is a norm in the case of $X$ being a normed $\FF$-vector space
since, for all $x=(x_1,\ldots,x_{2^{\dimA}})$, $x'=(x_1',\ldots,x_{2^{\dimA}}') \in X$,
$\Vert x+x' \Vert \=< \Vert x \Vert_X + \Vert x' \Vert_X$ can be proved by the property
\begin{align*}
   \bigg(\sum_{i=1}^{2^{\dimA}} \Vert x_i+x_i'\Vert_{X_i}^p \bigg)^{\frac{1}{p}}
& \=< \bigg(\sum_{i=1}^{2^{\dimA}} \Vert x_i\Vert_{X_i}^p \bigg)^{\frac{1}{p}}
   + \bigg(\sum_{i=1}^{2^{\dimA}} \Vert x_i'\Vert_{X_i}^p \bigg)^{\frac{1}{p}}
\end{align*}
of the $\homo$-norm $\Vert\cdot\Vert_{X_i}$.
Thus, for each $x = (x_1,\ldots,x_{2^{\dimA}})\in X$, $a\in A$ and $b\in B$, we have
\begin{align*}
  \Vert a.x.b \Vert_X
& = \Vert (a.x_1.b,\ldots, a.x_{2^{\dimA}}.b) \Vert_X \\
& = \bigg(
   c
   \sum_{i=1}^{2^{\dimA}} \Vert a.x_i.b \Vert_{X_i}^p
 \bigg)^{\frac{1}{p}}
  = \bigg(
   c
   \sum_{i=1}^{2^{\dimA}} \bfvert a\bfvert^p \Vert  x_i \Vert_{X_i}^p \Vert b \Vert_{B,p}^p
 \bigg)^{\frac{1}{p}} \\
& = \bfvert a\bfvert \bigg(
  c
   \sum_{i=1}^{2^{\dimA}} \Vert  x_i \Vert_{X_i}^p
 \bigg)^{\frac{1}{p}} \Vert b \Vert_{B,p}
  = \bfvert a\bfvert \Vert x \Vert_X  \Vert b \Vert_{B,p}.
\end{align*}
Therefore, $X$ is a $\homo$-normed $(A,B)$-bimodule.
\end{proof}

\begin{notation} \rm
Fixing a disjoint union $\II_A = \bigcup\limits_{i=1}^{\dimA} \II_i$ of $\II_A$.
If $X_1=X_2=\cdots=X_{2^{\dimA}} = N$ in Lemma \ref{lemm:dir sum},
then $\bigoplus\limits_{i=1}^{2^{\dimA}} X_i$ is written as $N^{\oplus_p 2^{\dimA}}$ for simplicity.
\end{notation}

Next, we define the category $\scrN_{\homo}^p$.

\begin{definition}[normed module category] \rm \label{def:Np}
A {\defines $\homo$-normed module category $\scrN_{\homo}^p$} of $A$
is a class of triples which are of the form $(N,v,\delta)$, where:
\begin{enumerate}[label=($\mathscr{N}$\arabic*)]
  \item $N$ is a $\homo$-normed $(A,B)$-bimodule;  \label{Np1}
  \item $v$ is an element in $N$ with $\Vert v\Vert_M \=< \mu(\II_{\alg})$ such that there is an $(A,B)$-homomorphism $\Pc: B^{\times I} \to N$ in ${_A}\Modcat_B$ with $\Pc((1_B)_{I})=v$, here, $(1_B)_{I} := (1_B)_{1\times I}$ is an element in the Cartesian product $B^{\times I} = \{ (b_i)_{1\times I} := (b_i)_{i\in I} \mid b_i \in B\}$ whose any component is the identity $1_B$ of $B$; \label{Np2}
  \item $\delta: N^{\oplus_p 2^{\dimA}} \to N$ is both a bounded $\FF$-linear map and an $(A,B)$-homomorphism
    (i.e., both a left $A$-homomorphism and a right $B$-homomorphism) satisfying $h((v)_{1\times 2^{\dimA}}) = v$.
    By the boundedness, it is clear that for any Cauchy sequence
    $\{x_i\}_{i\in \NN}$ in the completion $\w{N^{\oplus_p 2^{\dimA}}} = \w{N}^{\oplus_p 2^{\dimA}}$
    of $N^{\oplus_p 2^{\dimA}}$, $\delta(\invlim x_i) = \invlim \delta(x_i)$ holds.  \label{Np3}
\end{enumerate}
And for any two triples $(N,v,\delta)$ and $(N',v',\delta')$ in $\scrN_{\homo}^p$,
we define the morphism $(N,v,\delta) \to (N',v',\delta')$ to be the $(A,B)$-homomorphism
$\theta: N \to N'$ such that the following conditions hold.
\begin{enumerate}[label=($\mathscr{H}$\arabic*)]
  \item $\theta(v) = v'$;
    \label{H1}
  \item
the following diagram
\[
\xymatrix@R=1.5cm@C=1.5cm{
 N^{\oplus_p 2^{\dimA}}
\ar[r]^{\delta}
\ar[d]_{\theta^{\oplus 2^{\dimA}}
= \left(\begin{smallmatrix}
\theta & & \\
& \ddots & \\
& & \theta
\end{smallmatrix}
\right)_{2^{\dimA} \times 2^{\dimA}}
}
 &
 N
\ar[d]^{\theta}
\\
N'^{\oplus_p 2^{\dimA}}
\ar[r]_{\delta'}
 &
N'
}
\]
commutes.
    \label{H2}
\end{enumerate}
\end{definition}

\subsubsection{Banach module categories} \label{subsubsect:Ap}

Let $N$ be a $\homo$-normed $(A,B)$-bimodule. A {\defines Cauchy sequence} in $N$ is a sequence $\{x_u\}_{u=1}^{+\infty}$ such that for each $\epsilon \in \RR^+$, there exists $U \in \NN$ such that $\Vert x_{u_1} - x_{u_2} \Vert_N < \epsilon$ holds for all $u_1, u_2 > U$. Obviously, the sum of two Cauchy sequences is also a Cauchy sequence. In particular, if a Cauchy sequence $\{x_u\}_{u=1}^{+\infty}$ has a limit in $N$, i.e., there is an element $x\in N$ such that $\lim\limits_{u\to+\infty} x_u = x$, then $x$ is also a projective limit $x = \underleftarrow{\lim} x_u$ of $\{x_u\}_{u=1}^{+\infty}$, cf. \cite[Chapter 5, Section 5.2]{R1979}.
We call the {\defines completion} of $N$, say $\w{N}$, is the quotient $N^{\times \NN^+}/[0]$ obtained by $(A,B)$-bimodule
\[ N^{\times \NN^+} := \{ (x_1,x_2,\ldots) := \{x_u\}_{u=1}^{+\infty} \mid
  \{x_u\}_{u=1}^{+\infty} \text{~is a Cauchy sequence in~} N \} \]
modulo $[0]:=\{\{x_u\}_{u=1}^{+\infty} \in N^{\times \NN^+} \mid \{x_u\}_{u=1}^{+\infty} \sim \{0\}_{u=1}^{+\infty}\}$.
Here,
\begin{itemize}
  \item[(1)] the left $A$-action $A \times N^{\NN^+} \to N^{\NN^+}$ is defined as
    $a.(x_1,x_2, \ldots) := (a.x_1, a.x_2, \ldots)$;
  \item[(2)] the right $B$-action $N^{\NN^+} \times B \to N^{\NN^+}$ is defined as
    $(x_1,x_2,\ldots).b := (x_1.b, x_2.b, \ldots)$;
  \item[(3)] and the equivalence relation ``$\sim$'' is defined as
    \[ \{x_u\}_{u=1}^{+\infty} \sim \{y_u\}_{u=1}^{+\infty}
    :\Leftrightarrow \underleftarrow{\lim} (x_u-y_u) = 0.\]
\end{itemize}

\begin{definition}[Banach module] \rm \label{def:Banach}
A $\homo$-normed module $N$ is called a {\defines complete $\homo$-normed $(A,B)$-module} or a {\defines Banach $(A,B)$-module} if any Cauchy sequence $\{x_u\}_{u=1}^{+\infty} \in N^{\NN^+}$ has a limit in $N$, i.e., the map $f: \w{N} \to N, (x_1,x_2,\ldots) \mapsto \underrightarrow{\lim} x_u$ is an isomorphism of $(A,B)$-bimodules.
\end{definition}

For simplicity, we use $x \in \w{N}$ to represent the Cauchy sequence $(x,x,\ldots)$, then the homomorphism $f$ in Definition \ref{def:Banach} induces $N=\w{N}$, which can be viewed as a definition of Banach module.

\begin{definition}[Banach module category] \rm \label{def:Ban mod cat}
A {\defines Banach module category $\scrA_{\homo}^p$} of $A$ is a full subcategory $\scrN_{\homo}^p$ of $\scrA_{\homo}^p$ containing all objects $(N,v,\delta)$ with completed $N$.
\end{definition}

\subsection{Elementary simple functions}

A {\defines $\homo$-function} defined on $\II_A ~\mathop{\simeq}\limits^{1-1}~ \II^{\times \dimA}$ is a map $f: \II_A \to B$.
If $A$ and $B$ are normed tensor rings, then one can obtain two topologies defined on $A$ and $B$ by norms, respectively. Thus, we can define a $\homo$-function $f:\II_A \to B$ is {\defines continuous} if the preimage of any open subset of $\Ima(f)$ is an open subset of $\II_A$.
We do not differential between $\homo$-functions $f_1$ and $f_2$ if $f_1 \mathop{=}\limits^{\text{a.e.}} f_2$ (i.e., if $\mu_{\II_A}(\{f_1(x) \ne f_2(x) \mid x\in\II_A\}) = 0$).
A $\homo$-function $f: \II_A \to B$ is called a {\defines simple $\homo$-function} if its image $\Ima(f)$ is a finite subset of $B$.
All functions in this paper are $\homo$-function for simplicity.

\subsubsection{\texorpdfstring{$(A,B)$-bimodule $\bfS_{\homo}(\II_A)$}{(A,B)-bimodules S(IA)}}

\begin{definition} \rm
An {\defines elementary simple function} is a function
\[f: \II_A \to B, ~ \sum_{i=1}^t k_i\mathbf{1}_{I_i} ~ (k_1,\ldots, k_t \in \FF) \]
such that the following conditions hold.
\begin{itemize}
  \item[(1)]
    The set $I_i$ is a Cartesian product $I_i = I_{i,1} \times \cdots I_{i,\dimA}$, and for any $1\=< j\=< \dimA$,
    $I_{ij}$ is a subset of $\II = [c,d]_{\FF}$ which is one of the following forms:
    \begin{itemize}
      \item[(a)] $(c_{ij}, d_{ij})_{\FF} := \{k\in \FF \mid c_{ij} \prec k \prec d_{ij} \}$;
      \item[(b)] $[c_{ij}, d_{ij})_{\FF} := \{k\in \FF \mid c_{ij} \preceq k \prec d_{ij} \}$;
      \item[(c)] $(c_{ij}, d_{ij}]_{\FF} := \{k\in \FF \mid c_{ij} \prec k \preceq d_{ij} \}$;
      \item[(d)] $[c_{ij}, d_{ij}]_{\FF} := \{k\in \FF \mid c_{ij} \preceq k \preceq d_{ij} \}$,
    \end{itemize}
    where $a \preceq c_{ij} \preceq d$;

  \item[(2)]
    For each subset $S\subseteq A$, $\mathbf{1}_S$ is the function
    \[\mathbf{1}_S: A \to B, ~ a \mapsto
      \begin{cases}
        1_B, & \text{if~} a\in S; \\
        0_B, & \text{otherwise},
      \end{cases} \]
    and for all $1 \=< i \ne j \=< t$, $I_i\cap I_j = \varnothing$ holds
    ($1_B$ and $0_B$ are identity and zero in $B$).
\end{itemize}
\end{definition}

Let $\bfS(\II_A)$ be the set of all elementary simple functions, then the following lemma shows that $\bfS(\II_A)$ as an $\FF$-vector space with the homomorphism $\homo: A\to B$ induces a $(A,B)$-bimodule.

\begin{lemma} \label{lemm:bfS}
The set $\bfS(\II_A)$ of all elementary simple functions defined on $\II_A$ is an $\FF$-vector space.
Furthermore, $\bfS(\II_A)$ equipped with the left $A$-action
\[ A \times \bfS(\II_A) \to \bfS(\II_A), ~ (a,f) \mapsto a.f := (\homo(a)f: x\mapsto \homo(a)f(x))  \]
and the right $B$-action
\[ \bfS(\II_A) \times B \to \bfS(\II_A), ~ (f,b) \mapsto f.b := (fb: x\mapsto f(x)b), \]
say $\bfS_{\homo}(\II_A)$, is an $(A,B)$-bimodule.
\end{lemma}

\begin{proof}
For each $a, a_1, a_2\in A$, $b'\in B$, $f,f_1,f_2\in \bfS_{\homo}(\II_A)$, and $x\in\II_A$, we have:
\begin{itemize}
  \item[(1)] $((a_1+a_2).f)(x) = \homo(a_1+a_2)f(x) = (\homo(a_1)+\homo(a_2))f(x) = \homo(a_1)f+\homo(a_2)f(x) = (a_1.f+a_2.f)f(x)$ (see \ref{1M});
  \item[(2)] $(a.(f_1+f_2))(x) = \homo(a)(f_1+f_2)(x) = \homo(a)f_1(x)+\homo(a)f_2(x) = (a.f_1 + a.f_2)(x)$ (see \ref{2M});
  \item[(3)] $((a_1a_2).f)(x) = \homo(a_1a_2)f(x) = \homo(a_1)\homo(a_2)f(x) = \homo(a_1)(\homo(a_2)f(x)) = (a_1.(a_2.f))(x)$ (see \ref{3M});
  \item[(4)] $(1_A.f)(x) = \homo(1_A)f(x) = 1_Bf(x) = f(x)$ (see \ref{4M});
  \item[(5)] and, for any $\lambda\in \FF$, $(\lambda a).f = \homo(\lambda a)f=
    \begin{cases}
      (\lambda (\homo(a)f))(x) = \lambda (a.f)(x) \\
      \homo(a) (\lambda f(x)) = (a.(\lambda f))(x)
    \end{cases}
  $ (see \ref{5M}).
\end{itemize}
Thus, $\bfS_{\homo}(\II_A)$ is a left $A$-module. One can check that \ref{M1}--\ref{M5} holds, and thus, $\bfS_{\homo}(\II_A)$ is a right $A$-module. Finally, we have $(a.(f.b'))(x) = \homo(a)(f(x)b') = (\homo(a)f(x))b' = ((a.f).b')(x)$, it follows that $\bfS_{\homo}(\II_A)$ is an $(A,B)$-bimodule as required.
\end{proof}

\begin{proposition} \label{prop:normed bfS}
The $(A,B)$-bimodule $\bfS_{\homo}(\II_A)$ with the map
\[ \Vert \cdot \Vert_p : \bfS_{\homo}(\II_A) \to \RR^{\>=0}, \
~ f = \sum_{i=1}^t b_i \mathbf{1}_{I_i}
\mapsto \bigg( (\Vert b_i \Vert_{B,p}^p \mu_{\II_A}(I_i))^p \bigg)^{\frac{1}{p}} \]
is a $\homo$-normed $(A,B)$-bimodule.
\end{proposition}

\begin{proof}
We need to show that \ref{N1} and \ref{N2} hold. However, the difficulty of the proof of \ref{N1} lies in the proof of the triangle inequality, and \ref{N2} can be proved by using the fact $a.f.b = T(a)fb$. Thus, we only prove the triangle inequality in this proof.

For two arbitrary functions $f = \sum_i b_i\mathbf{1}_{I_i}$ and $g = \sum_j b_j'\mathbf{1}_{I_j'}$
(here, if $i\ne \imath$, then $I_i\cap I_{\imath} = \varnothing$; and if $j \ne \jmath$, then $I_j\cap I_{\jmath} = \varnothing$), we have
\[ f+g = \sum_i b_i\mathbf{1}_{I_i \backslash \bigcup_j I_j'}
+ \sum_j b_j'\mathbf{1}_{I_j' \backslash \bigcup_j I_j}
+ \sum_{I_i\cap I_j'=\varnothing} (b_i \mathbf{1}_{I_i\cap I_j'}
+ b_i' \mathbf{1}_{I_i\cap I_j'}). \]
Then
\[ \Vert f+g\Vert_p = ({\color{red}R} + {\color{green}G}+ {\color{blue}B})^{\frac{1}{p}} , \]
where
\begin{align*}
{\color{red}R}
  &= {\color{red} \sum_i \Vert b_i \Vert_{B,p}^p \mu_{\II_A}(I_i \backslash \bigcup_j I_j')^p } ; \\
{\color{green}G}
  &= {\color{green}\sum_j \Vert b_j' \Vert_{B,p}^p \mu_{\II_A}(I_j' \backslash \bigcup_i I_i)^p }; \\
{\color{blue}B}
  &= {\color{blue} \sum_{I_i \cap I_j' = \varnothing}
     (\Vert b_i \Vert_{B,p}^p + \Vert b_i' \Vert_{B,p}^p)
     \mu_{\II_A}(I_i \cap I_j')^p
     }.
\end{align*}
By the discrete Minkowski inequality, we have
\begin{align*}
  \Vert f \Vert_p + \Vert g \Vert_p
&= \bigg( \sum_i \Vert b_i\Vert_{B,p}^p \mu_{\II_A}(I_i)^p \bigg)^{\frac{1}{p}}
 + \bigg( \sum_i \Vert b_j'\Vert_{B,p}^p \mu_{\II_A}(I_j)^p \bigg)^{\frac{1}{p}} \\
& \>= \bigg( \sum_i \Vert b_i\Vert_{B,p}^p \mu_{\II_A}(I_i)^p
  + \sum_i \Vert b_j'\Vert_{B,p}^p \mu_{\II_A}(I_j)^p \bigg)^{\frac{1}{p}} =: \mathfrak{N}.
\end{align*}
Note that $\mu_{\II_A}(X\cup Y) = \mu_{\II_A}(X)+\mu_{\II_A}(Y)$ holds for all $X, Y \subseteq \II_A$ with $X\cap Y = \varnothing$, we have $\mu(X\cap Y)^p = (\mu_{\II_A}(X)+\mu_{\II_A}(Y))^p \>= \mu_{\II_A}(X)^p+\mu_{\II_A}(Y)^p$, then
\[ \mu_{\II_A}(I_i)^p \>=
   \mu_{\II_A}(I_i \backslash \bigcup\nolimits_j I_j')^p
 + \mu_{\II_A}(I_i \cap \bigcup\nolimits_j I_j')^p. \]
It admits
\begin{align}
        \sum\nolimits_i \Vert b_i\Vert_{B,p}^p\mu(I_i)^p
\>=\ &  \sum\nolimits_i \Vert b_i\Vert_{B,p}^p\mu_{\II_A}\big(I_i\backslash\bigcup\nolimits_j I_j'\big)^p
        + \sum\nolimits_i \Vert b_i\Vert_{B,p}^p\mu\big(I_i\cap\bigcup\nolimits_j I_j'\big)^p
        \nonumber \\
  =\ &  {\color{red}R}
        + \sum\nolimits_i \Vert b_i\Vert_{B,p}^p
        \bigg(\mathop{\sum\nolimits_{j}}\limits_{I_i\cap I_j'\ne\varnothing}\mu(I_i \cap I_j')\bigg)^p
        \nonumber \\
  \mathop{\>=}\
     &  {\color{red}R} + \sum_{I_i\cap I_j'\ne\varnothing} \Vert b_i\Vert_{B,p}^p\mu(I_i\cap I_j')^p \nonumber \\
  (\text{write it as~} & {\color{red}R} + B_1), \nonumber
\end{align}
and, similarly, admits
\begin{align}
       \sum\nolimits_j \Vert b_j'\Vert_{B,p}^p\mu(I_j')^p
\>=\ & {\color{green}G} + \sum_{I_j'\cap I_i\ne\varnothing} \Vert b_j'\Vert_{B,p}^p\mu(I_j'\cap I_i)^p \nonumber \\
  (\text{write it as~} & {\color{green}G} + B_2). \nonumber
\end{align}
Clearly, ${\color{blue} B}=B_1+B_2$, and so, $\mathfrak{N} = (({\color{red}R} + B_1) + ({\color{green}G} + B_2))^{\frac{1}{p}} = ({\color{red}R}+ {\color{green}G} + {\color{blue}B})^{\frac{1}{p}}$ $=\Vert f+g\Vert_p$ as required.
\end{proof}

\subsubsection{\texorpdfstring{$(A,B)$-bimodule $E_u$}{(A,B)-bimodule Eu}} \label{subsubsect:Eu}

Now, assume that there is an element $\xi\in (c,d)_{\FF}$ such that two maps order-preserving bijections $\kappa_c: [c,d]_{\FF} \to [c,\xi]_{\FF}$ and $\kappa_d: [c,d]_{\FF} \to [\xi,d]_{\FF}$ exist,
and define
\begin{align}\label{formula:E0}
  E_0 := \{ f: \II_A \to B \text{~is a~$\homo$-function~} \mid f(x)=b \text{~is a constant in~} B \}.
\end{align}
The following lemma shows that $E_0$ is an $(A,B)$-bimodule.

\begin{lemma} \label{lemm:E0}
Left $A$-action and right $B$-action
\[ A \times E_0 \to E_0, (a, f(x)) \mapsto a.f(x):=\homo(a)f(x)\]
and
\[ E_0 \times B \to E_0, (f(x), b) \mapsto f(x).b:=f(x)b \]
admit that $E_0$ is an $(A,B)$-bimodule. Furthermore, $E_0\cong B$.
\end{lemma}

\begin{proof}
One can check that $E_0$ is an $(A,B)$-bimodule by a method similar to the proof of Lemma \ref{lemm:bfS}.
On the other hand, the corresponding $h: E_0 \to B, (f: \II_A \to \{b\}) \mapsto b$ satisfies
\[ h(f_1+f_2: \II_A \to \{b_1+b_2\}) = b_1+b_2 = h(f_1)+h(f_2)
~ (\forall f_1,f_2\in E_0) \]
and
\[ h(a.f.b') = \homo(a)bb' = a.h(f).b' ~ (\forall a\in A, b'\in B \text{~and~} f:\II_A\to\{b\} \in E_0).  \]
Thus, $h$ is a homomorphism between two $(A,B)$-bimodules.
It is clear that $h$ is a bijection. Then $h$ is an isomorphism.
\end{proof}

Assume $\basis{A} = \{e_{A,i} \mid 1 \=< i \=< \dim_{\FF} A =\dimA\} = (\Q_A)_{\>= 0}$.
Then any element $x\in\II_A$ has a decomposition
\[ x = \sum_{i=1}^{\dimA} k_i e_{A,i}, ~ k_1,\ldots, k_{\dimA} \in \FF. \]
For a sequence
\[\pmb{f}=(f_{(\sigma_1,\ldots,\sigma_{\dimA})} : \II_A \to B)_{(\sigma_1,\ldots,\sigma_{\dimA}) \in \{c,d\}^{\times \dimA}}\]
of any $2^{\dimA}$ functions, we define $\gamma_{\xi}(\pmb{f})$ is the function
\begin{align}\label{jux map}
\gamma_{\xi} (\pmb{f})(k_1,\ldots, k_{\dimA})
 = \sum_{(\sigma_1,\ldots,\sigma_{\dimA}) \in \{c,d\}^{\times \dimA}}
   \mathbf{1}_{ \Pi_{(\sigma_1,\ldots,\sigma_{\dimA})} }
   f_{(\sigma_1,\ldots,\sigma_{\dimA})}
     (
       \kappa_{\sigma_1}^{-1}(k_1), \ldots, \kappa_{\sigma_{\dimA}}^{-1}(k_{\dimA})
     )
\end{align}
\[(k_1 \ne \xi, \ldots, k_{\dimA} \ne \xi), \]
where $\Pi_{(\sigma_1,\ldots,\sigma_{\dimA})} := \displaystyle
\prod\limits_{i=1}^{\dimA} \kappa_{\sigma_i}(\II)$ satisfies that
$\Pi_{(\sigma_1,\ldots,\sigma_{\dimA})} \cap \Pi_{(\tilde{\sigma}_1,\ldots,\tilde{\sigma}_{\dimA})}
= \varnothing$ holds for all $(\sigma_1,\ldots,\sigma_{\dimA}) \ne (\tilde{\sigma}_1,\ldots,\tilde{\sigma}_{\dimA})$.

Let $\Func(\II_A)$ be the set of all functions $A \to B$, then it is an $(A,B)$-bimodule, and $\gamma_{\xi}$ can be seen as a map
\[ \gamma_{\xi}: \Func(\II_A)^{\oplus 2^{\dimA}} \to \Func(\II_A). \]
In general, we do not define a norm on $\Func(\II_A)$ (such as the set of all functions $f: [0,1] \to \RR$), thus $\Func(\II_A)^{\oplus 2^{\dimA}}$ is only a direct sum of $2^{\dimA}$ $(A,B)$-bimodules $\Func(\II_A)$ it the above map.
Moreover, $E_0 \subseteq \Func(\II_A)$ is clear, then we have a restriction $\gamma_{\xi}|_{E_0}: E_0^{\oplus_p 2^{\dimA}} \to \Func(\II_A)$. The map $\gamma_{\xi}$ is called a {\defines juxtaposition map} in \cite{Lei2023FA}.

\begin{example} \rm
Consider the case of $A=\RR^3$ being a semi-simple algebra and $B=\RR$ being a field, and let $\II_A = [0,1]^{\times 3}$, $0<\xi<1$, and $\color{violet}f_{000}(x,y,z)$, $\color{red}f_{100}(x,y,z)$, $\color{green}f_{110}(x,y,z)$, $\color{cyan}f_{010}(x,y,z)$, $\color{green}f_{001}(x,y,z)$, $\color{cyan}f_{101}(x,y,z)$, $\color{violet}f_{111}(x,y,z)$, $\color{red}f_{011}(x,y,z)$ be eight functions in $\Func(\II_A)$. In \Pic \ref{fig:jux map}, we draw the domains of $f_{011}(x,y,z)$ and $f_{111}(x,y,z)$ (see the cubes marked by $\mathrm{Dom}(f_{011})$ and $\mathrm{Dom}(f_{111})$), and the domains of other seven functions are ignored for simplicity.
\begin{figure}[htbp]
  \centering
\includegraphics[width=16.5cm]{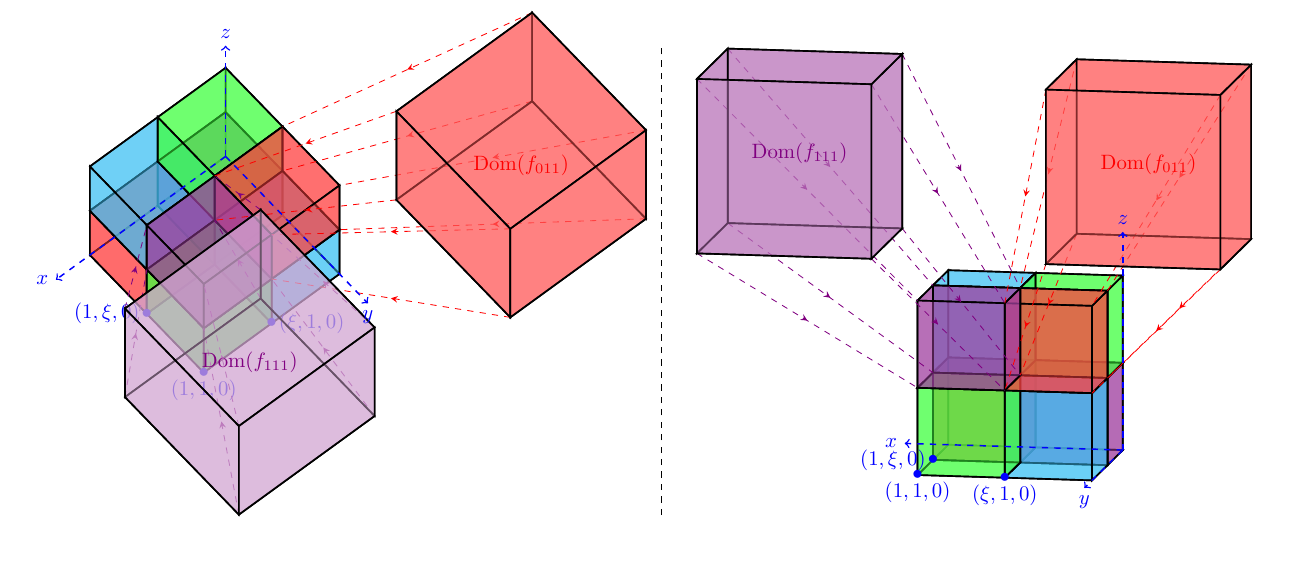} \\
  \caption{Juxtaposition map}
  \label{fig:jux map}
\end{figure}
Then $(f_{000}, f_{100}, \ldots, f_{011})$ is an element in $\Func(\II_A)^{\oplus 8}$, and $\gamma_{\xi}$ sends it to a function $\gamma_{\xi}(f_{000}, f_{100}, \ldots, f_{011})$ whose domain is a cube $[0,1]^{\times 3} \subseteq A = \RR^3$, see the cube with a side length of $1$ which is formed by splicing $8$ small cubes as shown in \Pic \ref{fig:jux map}. The dashed arrow ``$\color{red}{\xymatrix{\ar@{-->}[r]&}}$'' shown in this figure represents applying juxtaposition map $\gamma_{\xi}$ to the function $f_{011}(x,y,z)$.
\end{example}

We define
\[ E_1 := \Ima(\gamma_{\xi}|_{E_0}), \]
and for each $u\in\NN$, we define
\[ E_{u+1} := \Ima(\gamma_{\xi}|_{E_u}). \]
Then we have the following lemma.

\begin{lemma} \label{lemm:Eu}
All $E_u$ are normed $(A,B)$-bimodules. Furthermore, for each $u\in \NN$, $\gamma_{\xi}|_{E_u}$ provide an isomorphism $E_u^{\oplus_p 2^{\dimA}} \cong E_{u+1}$ between two $(A,B)$-bimodules.
\end{lemma}

\begin{proof}
For any $u\in \NN$, one can check that $E_u$ is an $(A,B)$-bimodule in a way similar to the proof of Lemma \ref{lemm:E0}. By the definition of $E_u$, it is clear that $\gamma_{\xi}|_{E_u}$ is an epimorphism between two modules. Next, we show that $\gamma_{\xi}$ is injective.
To do this, take two functions $\pmb{f}=(f_1,\ldots,f_{\dimA})$ and $\pmb{g}=(g_1,\ldots,g_{\dimA})$ in $E_u$ such that $\gamma_{\xi}|_{E_u}(\pmb{f}) = \gamma_{\xi}(\pmb{f}) = \gamma_{\xi}(\pmb{g}) = \gamma_{\xi}|_{E_u}(\pmb{g})$ holds.
By (\ref{jux map}), $\gamma_{\xi}(\pmb{f})$ and $\gamma_{\xi}(\pmb{g})$ are of the forms
\[ \gamma_{\xi}(\pmb{f}) = \sum_{I_i} \mathbf{1}_{I_i}\cdot f_i(\kappa_1^{-1}(k_1), \ldots, \kappa_{\dimA}^{-1}(k_1)) \]
and
\[ \gamma_{\xi}(\pmb{g}) = \sum_{I_i} \mathbf{1}_{I_i}\cdot g_i(\kappa_1^{-1}(k_1), \ldots, \kappa_{\dimA}^{-1}(k_1)), \]
respectively. Here, $I_i \cap I_j = \varnothing$ holds for all $i\ne j$. Then we have
\[ \gamma_{\xi}(\pmb{f}-\pmb{g}) = \sum_{I_i} \mathbf{1}_{I_i}\cdot \big(f_i-g_i\big)(\kappa_1^{-1}(k_1), \ldots, \kappa_{\dimA}^{-1}(k_1)) = 0. \]
It follows that $f_i=g_i$ holds for all $(\kappa_1^{-1}(k_1), \ldots, \kappa_{\dimA}^{-1}(k_1))$,
and then we have $\pmb{f}=\pmb{g}$ as required.

Take $f \in E_0$, we have $\Ima(f) = b$ for some $b\in B$, then $\Vert f \Vert_{E_0} := \Vert b \Vert_{B,b}$ induces a norm of $f$. Then
\[\Vert\cdot\Vert_{E_0}: E_0 \to \RR^{\>=0}, ~(f:\II_A\to\{b\}) \mapsto \Vert b\Vert_{B,p}\]
is a norm defined on $E_0$. Thus, $E_0$ is a normed $(A,B)$-bimodule.
Take $f \in E_u$ ($u\>= 1$), then by the definition of $E_u$, $f$ can be written as a finite sum
\[ f = \sum_{i=1}^{2^{\dimA}} f_i \mathbf{1}_{I_i}, \]
where all functions $f_i$ lie in $E_{u-1}$ and $\II_{A} = \bigcup_i I_i$ is a disjoint union.
By Lemma \ref{lemm:dir sum}, the map \[ \Vert\cdot\Vert_{E_u}: E_u \to \RR^{\>=0},
~f \mapsto \displaystyle
\bigg(\sum_{i=1}^{2^{\dimA}}
  \bigg(
    \frac{\mu_{\II_A}(I_i)}{\mu_{\II_A}(\II_A)}
  \bigg)^p
\Vert f_i \Vert^p
\bigg)^{\frac{1}{p}}
\]
is a norm defined on $E_u$.
\end{proof}

It is clear that $E_u \subseteq E_{u+1}$ for any $u\in\NN$ by the definition of $E_u$. The following lemma shows
$E_0 \mathop{\longrightarrow}\limits^{\subseteq}
   E_1 \mathop{\longrightarrow}\limits^{\subseteq}
   E_2 \mathop{\longrightarrow}\limits^{\subseteq}
   \cdots
       \mathop{\longrightarrow}\limits^{\subseteq}
   E_u \mathop{\longrightarrow}\limits^{\subseteq}
   \cdots \subseteq \bfS_{\homo}(\II_A).$

\begin{lemma} \label{lemm:Eu subset bfS}
For any $u\in\NN$, we have $E_u \subseteq \bfS_{\homo}(\II_A)$.
\end{lemma}

\begin{proof}
Let $\basis{}=\{ \mathbf{1}_X \mid X \subseteq \II_A \}$. Then $\basis{}$ is a generator set of $\bfS_{\homo}(\II_A)$, and we obtain a free precover $\Pc: B^{\oplus \basis{}} \to \bfS_{\homo}(\II_A)$, $(b_X\mathbf{1}_{X})_{X\subseteq \II_A} \mapsto \sum_{X\subseteq \II_A} b_X\mathbf{1}_{X}$ of $\bfS_{\homo}(\II_A)$.
By Lemma \ref{lemm:Eu}, we have $E_u \cong E_{u-1}^{\oplus 2^{\dimA}} \cong \cdots \cong E_0^{\oplus u2^{\dimA}}$ holds for all $u\in \NN$, and by the definition of $E_0$ (see (\ref{formula:E0})), we have $E_0\cong B$. Thus, $E_u \cong B^{\oplus u2^{\dimA}}$.
On the other hand, $E_u \subseteq B^{\oplus \basis{}}$, then there exists an embedding $\mathrm{emb}: B^{\oplus u2^{\dimA}} \hookrightarrow B^{\oplus\basis{}}$ induced by $B^{\oplus u2^{\dimA}} \cong E_u \subseteq B^{\oplus\basis{}}$. Thus, we obtain an $(A,B)$-homomorphism
\[ \Pc_{\bfS_{\homo}(\II_A)} := \Pc\compos\mathrm{emb} :
\xymatrix{ B^{\oplus u2^{\dimA}} \ar@{^(->}[r]^{\mathrm{emb}} &  B^{\oplus\basis{}} \ar[r]^{\Pc} & \bfS_{\homo}(\II_A) } \]
which admits $E_u \subseteq \bfS_{\homo}(\II_A)$.
\end{proof}

\subsubsection{\texorpdfstring{$\w{\bfS_{\homo}(\II_A)} \cong \protect\underrightarrow{\lim} E_u$}{hat S(IA) = lim Eu}}

Let ${_A\Nor_B}$ be the category of normed $(A,B)$-bimodules and $(A,B)$-homomorphism between them. By Lemmas \ref{lemm:bfS}, \ref{lemm:E0}, and \ref{lemm:Eu}, we get that $\bfS_{\homo}(\II_A)$, $\w{\bfS_{\homo}(\II_A)}$, and all $\FF$-vector spaces $E_u$ ($u\in\NN$) are $(A,B)$-bimodules. Let ${_A\Ban_B}$ be the category of Banach $(A,B)$-bimodules and $(A,B)$-homomorphism between them.
Then ${_A\Ban_B}$ is a full subcategory of ${_A\Nor_B}$ and $\w{\bfS_{\homo}(\II_A)}$ is an object in ${_A\Ban_B}$.
Now, consider all $(A,B)$-homomorphisms $\varphi_{ij}: E_i \to E_j$ ($i \=< j$) which are given by $E_i \subseteq E_j$, we obtain a direct system $((E_i)_{i\in\NN}, (\varphi_{uv})_{u\=< v})$. The following result provide a description of the completion $\w{\bfS_{\homo}(\II_A)}$ of $\bfS_{\homo}(\II_A)$ by using this direct system.

\begin{lemma} \label{lemm:limits}
Assume $\FF$ is completed and let $(\alpha_i : E_i \to \w{\bfS_{\homo}(\II_A)})_{i\in\NN}$ be a family of $(A,B)$-homomorphisms given by $E_i \subseteq \w{\bfS_{\homo}(\II_A)}$.
Then, in the the sense of $(\alpha_i)_{i\in\NN}$ to be insertion morphisms, the inductive limit of the direct system $((E_i)_{i\in\NN}, (\alpha_{uv})_{u\=< v})$ in ${_A\Ban_B}$ is isomorphic to $\w{\bfS_{\homo}(\II_A)}$, i.e.,
\[  \w{\bfS_{\homo}(\II_A)} \cong \dirlim~E_u. \]
Furthermore, $\w{\bfS_{\homo}(\II_A)}$ is a normed $(A,B)$-bimodule whose norm $\Vert\cdot\Vert_{\w{\bfS_{\homo}(\II_A)}}$ is naturally induced by the norm $\Vert\cdot\Vert_{E_u}$ of $E_u$, i.e.,
\[ \Vert\cdot\Vert_{\w{\bfS_{\homo}(\II_A)}} \simeq \dirlim~\Vert\cdot\Vert_{E_u} \]
\end{lemma}

\begin{proof}
Let $X$ a Banach $(A,B)$-bimodule in ${_A\Ban_B}$ such that there is a family $(f_i: E_i \to X)_{i\in\NN}$ of $(A,B)$-homomorphism satisfying $f_i = f_j \alpha_{ij}$ for all $i\=< j$. Now, we define $\theta: \w{\bfS_{\homo}(\II_A)} \to X$ in the following way.

For any $x \in \w{\bfS_{\homo}(\II_A)}$, there exists a Cauchy sequence $\{x_t\}_{t\in\NN}$ in $\bigcup_{i\in\NN} E_i$ such that $\{\Vert x_t - x \Vert\}_{t\in \NN}$ is a monotonically decreasing Cauchy sequence in $\RR^{\>= 0}$ with $\underleftarrow{\lim} \Vert x_t - x \Vert = 0$. It follows that $\underleftarrow{\lim} x_t = x$.
Notice that each $x_t$ must lie in some $(A,B)$-bimodule $E_{u(t)}$ ($u(t)\in\NN$, and, clearly, $x_t \in E_u$ holds for all $u\>= u(t)$), then $x_t$ has a preimage $x_t'$ given by $\alpha_{u(t)}$. Define
\[ \theta(x) = \underleftarrow{\lim} f_{u(t)}(x_t), \]
and let $f: \bigcup_{i\in\NN} E_i \to X$ be the $(A,B)$-homomorphism induced by the direct system $((E_i)_{i\in\NN}, (\alpha_{uv})_{u\=< v})$, we immediately obtain
\[ \theta(x) = \underleftarrow{\lim} f|_{E_{u(t)}}(x_t) = \underleftarrow{\lim} f(x_t). \]
Then one can check that $\theta$ is well-defined since the projective limit is unique.
For each $j\>= i$, consider the following diagram, we have $\alpha_{ij}\alpha_j = \alpha_i$ and $f_j\alpha_{ij} = f|_{E_j}\alpha_{ij} = f_i = f|_{E_i}$.
\[\begin{tikzpicture}
\draw[rotate=15+  0][<-][line width=1pt] (2,0) arc(0:90:2) [red][dashed];
\draw[rotate=15+120][<-][line width=1pt] (2,0) arc(0:90:2);
\draw[rotate=15+240][->][line width=1pt] (2,0) arc(0:90:2);
\draw (2,0) node{$X$};
\draw[rotate=120] (2,0) node{$\w{\bfS_{\homo}(\II_A)}$};
\draw[rotate=240] (2,0) node{$E_j$};
\draw (0,0) node{$E_i$};
\draw[->][line width=1pt] (0.5,0) to (1.5,0);
\draw[rotate=120][->][line width=1pt] (0.5,0) to (1.5,0);
\draw[rotate=240][->][line width=1pt] (0.5,0) to (1.5,0);
\draw[rotate=60+  0] (2.4,0) node[red]{$\exists~\theta$};
\draw[rotate=60+120] (2.4,0) node{$\alpha_j$};
\draw[rotate=60+240] (2.4,0) node{$f_j$};
\draw[rotate= 0+  0] (1.0,0.3) node{$f_i$};
\draw[rotate= 0+120] (1.0,0.3) node{$\alpha_i$};
\draw[rotate= 0+120] (1.0,-0.3) node{$\subseteq$};
\draw[rotate= 0+240] (1.0,0.3) node{$\alpha_{ij}$};
\draw[rotate= 0+235] (1.0,-0.3) node{{$(i\=< j) \atop  \subseteq$}};
\end{tikzpicture}\]

We need show that $\theta$ is unique. To do this, assume that there is an $(A,B)$-homomorphism $\vartheta: \w{\bfS_{\homo}(\II_A)} \to X$ such that $f_i = \vartheta\alpha_i$ and $f_j=\vartheta\alpha_j$ holds for all $i \=< j$.
Then we have $\alpha_{u(t)}\theta(x_t) = f_{u(t)}(x_t) = \alpha_{u(t)}\vartheta(x_t)$, and then
\[ \alpha_{u(t)}(\theta(x_t)-\vartheta(x_t)) = 0. \]
Since all $\alpha_i$ are injective, we obtain $\theta(x_t)=\vartheta(x_t)$.
Furthermore, all $\FF$-linear maps are continuous by using the completion of $\w{\bfS_{\homo}(\II_A)}$, we have $\underleftarrow{\lim}\theta(x_t)=\underleftarrow{\lim}\vartheta(x_t)$, i.e., $\theta(x) = \vartheta(x)$ holds for all $x\in \w{\bfS_{\homo}(\II_A)}$.
Naturally, the formula $\Vert\cdot\Vert_{\w{\bfS_{\homo}(\II_A)}} \simeq \dirlim~\Vert\cdot\Vert_{E_u}$ can be induced by $\w{\bfS_{\homo}(\II_A)} \cong \dirlim~E_u$.
\end{proof}

\subsection{\texorpdfstring{Triples $(\bfS_{\homo}(\II_A), \mathbf{1}_{\II_A}, \gamma_{\xi})$
and $(\w{\bfS_{\homo}(\II_A)}, \mathbf{1}_{\II_A}, \w{\gamma}_{\xi})$}{Triples
(S(IA),1,gamma) and (hat S(IA),1 , hat gamma)}}

We will consider two triples $(\bfS_{\homo}(\II_A), \mathbf{1}_{\II_A}, \gamma_{\xi})$ and $(\w{\bfS_{\homo}(\II_A)}, \mathbf{1}_{\II_A}, \w{\gamma}_{\xi})$ in this subsection, which are important objects in $\scrN_{\homo}^p$.

\subsubsection{\texorpdfstring{$(\bfS_{\homo}(\II_A), \mathbf{1}_{\II_A}, \gamma_{\xi})$ as an object in $\scrN_{\homo}^p$}{(S(IA),1,gamma) as an object in Np}}

Let $\bfS_{\homo}(\II_A)^{\oplus_p 2^{\dimA}} = \bfS^{\oplus}$ and $\mathbf{1}=\mathbf{1}_A: A \to \{1_B\}$.
Recall the definition of $\gamma_{\xi}: \Func(\II_A)^{\oplus_p 2^{\dimA}} \to \Func(\II_A)$, it induce two maps $\gamma_{\xi}|_{\bfS^{\oplus}}: \bfS^{\oplus} \to \bfS_{\homo}(\II_A)$ and $\w{\gamma}_{\xi}|_{\bfS^{\oplus}}: \w{\bfS}^{\oplus} \to \w{\bfS_{\homo}(\II_A)}$, where the map $\w{\gamma}_{\xi}$ is obtained by the completion of $\gamma_{\xi}$. For simplicity, we do not differentiate between the notations $\gamma_{\xi}|_{\bfS}$ and $\gamma_{\xi}$ in this paper.

\begin{lemma} \label{lemm:Pc}
There is an $(A,B)$-homomorphism
$\Pc_{\bfS_{\homo}(\II_A)}: B^{\times I} \to \bfS_{\homo}(\II_A)$
sending $(1_B)_{1\times I}$ to $\mathbf{1}_{\II_A}$.
\end{lemma}

\begin{proof}
This is a direct corollary of Lemma \ref{lemm:Eu subset bfS}. To be more precise, we have $\mathbf{1}_{\II_A} \in E_0^{\oplus u2^{\dimA}} \cong B^{\oplus u2^{\dimA}} \cong E_u \subseteq B^{\oplus \basis{}}$, and it can be seen as a finite sum which has the following form
\[ \mathbf{1}_{\II_A} = \sum_i \mathbf{1}_{I_i}, \text{~where~} I_i \cap I_j =\varnothing~(\forall i\ne j), \bigcup_i I_i = \II_A. \]
Thus, the composition $\Pc_{\bfS_{\homo}(\II_A)} = \Pc~\mathrm{emb}$ given in the proof of Lemma \ref{lemm:Eu subset bfS} sends $(1_B)_{1\times u2^{\dimA}} \in \mathbf{1}_{\II_A}$ to the function $\mathbf{1}_{\II_A} \in \bfS_{\homo}(\II_A)$.
\end{proof}

\begin{proposition} \label{prop:bfS lies in Np}
The triple $(\bfS_{\homo}(\II_A), \mathbf{1}, \gamma_{\xi})$ is an object in $\scrN_{\homo}^p$.
\end{proposition}

\begin{proof}
First of all, by Lemma \ref{lemm:bfS} and Proposition \ref{prop:normed bfS}, we obtain that $\bfS_{\homo}(\II_A)$ is a $\homo$-normed $(A,B)$-bimodule. Thus, \ref{Np1} holds.
Second, by the $\homo$-norm $\Vert\cdot\Vert_p$ defined on $\bfS_{\homo}(\II_A)$ (see Proposition \ref{prop:normed bfS}), we have $\Vert \mathbf{1} \Vert_p = \Vert \mathbf{1}_{\II_A} \Vert_p = (\mu_{\II_A}(\II_A)^p)^{\frac{1}{p}} = \mu_{\II_A}(\II_A)$. In addition, Lemma \ref{lemm:Pc} provides an $(A,B)$-homomorphism $\Pc_{\bfS_{\homo}(\II_A)}: B^{\times I} \to \bfS_{\homo}(\II_A)$ sending $(1_B)_{1\times I}$ to $\mathbf{1}_{\II_A}$. Thus, we have \ref{Np2}.

Next, we prove \ref{Np3}. For any Cauchy sequence $\{\pmb{f}_t\}_{t\in\NN}$ in $\w{\bfS}^{\oplus}$, we need prove $\w{\gamma}_{\xi}(\invlim \pmb{f}_t) = \invlim \w{\gamma}_{\xi}(\pmb{f}_t)$ in this proof.
Here, $\w{\gamma}_{\xi}$ is an $(A,B)$-homomorphism $\w{\bfS}^{\oplus} \to \w{\bfS_{\homo}(\II_A)}$ induced by the completion of $\bfS_{\homo}(\II_A)$.
By Lemma \ref{lemm:Eu}, for each $u\in\NN$, $\gamma_{\xi}|_{E_u}$ is an $(A,B)$-isomorphism, then, by Lemma \ref{lemm:limits}, $\w{\gamma}_{\xi}$ is also an $(A,B)$-isomorphism. Therefore, we have
\[ \gamma_{\xi}(\invlim \pmb{f}_t)
\mathop{=}\limits^{\spadesuit} \w{\gamma}_{\xi}(\invlim \pmb{f}_t)
\mathop{=}\limits^{\clubsuit} \invlim \w{\gamma}_{\xi}(\pmb{f}_t)
\mathop{=}\limits^{\heart} \invlim \gamma_{\xi}(\pmb{f}_t)  \]
as required, where $\spadesuit$ is given by $\gamma_{\xi}: \w{\bfS}^{\oplus} \to \w{\bfS_{\homo}(\II_A)}$ being a restriction of $\w{\gamma_{\xi}}|_{\bfS_{\homo}(\II_A)}$,
$\clubsuit$ is given by $\w{\gamma}_{\xi}$ is an isomorphism,
and $\heart$ holds since there is an integer $u(t)\in\NN$ with $\w{\gamma}_{\xi}(\pmb{f}_t) = \gamma_{\xi}|_{E_u(t)}(\pmb{f}_t) = \gamma_{\xi}(\pmb{f}_t)$.
\end{proof}

\subsubsection{\texorpdfstring{$(\w{\bfS_{\homo}(\II_A)}, \mathbf{1}_{\II_A}, \w{\gamma}_{\xi})$ as an object in $\scrA_{\homo}^p$}{(hat S(IA), 1, hat gamma) as an object in Ap} }

Proposition \ref{prop:bfS lies in Np} shows that $(\bfS_{\homo}(\II_A), \mathbf{1}_{\II_A}, \gamma_{\xi})$ is an object in $\scrN^p_{\homo}$. Then the completion $\w{\bfS_{\homo}(\II_A)}$ of $\bfS_{\homo}(\II_A)$ induced a new triple $(\w{\bfS_{\homo}(\II_A)}, \mathbf{1}_{\II_A}, \w{\gamma}_{\xi})$ is also an object in $\scrN^p_{\homo}$. Recall the definition of $\scrA_{\homo}^p$ (see Definition \ref{def:Ban mod cat}), it is clear that $(\w{\bfS_{\homo}(\II_A)}, \mathbf{1}_{\II_A}, \w{\gamma}_{\xi})$ is also an object in $\scrA_{\homo}^p$.
In this paper, we want to know if it is an initial object in $\scrA_{\homo}^p$. Thus, we need consider the existence of homomorphism from $(\w{\bfS_{\homo}(\II_A)}, \mathbf{1}_{\II_A}, \w{\gamma}_{\xi})$ and the uniqueness of this homomorphism.

\begin{proposition} \label{prop:existence}
For any object $(N,v,\delta)$ in $\scrA^p$, we have
\[ \Hom_{\scrA^p_{\homo}}(
  (\w{\bfS_{\homo}(\II_A)}, \mathbf{1}_{\II_A}, \w{\gamma}_{\xi}),
  (N,v,\delta)
  ) \ne \varnothing. \]
\end{proposition}

\begin{proof}
For each $(N,v,\delta)$ in $\scrA_{\homo}^p$, since there is an $(A,B)$-homomorphism $\theta: B^{\times I} \to N$ with $\theta((1_B)_{1\times I}) = v$, then, by using the isomorphism $\eta: B \mathop{\to}\limits^{\cong} E_0$ given in Lemma \ref{lemm:E0}, the $(A,B)$-homomorphism $h: E_0 \to B^{\times I}$, $x\mapsto (\eta^{-1}(x))_{1\times I}$ induces a composition
\[ h \eta : \xymatrix{
B \ar[r]^{\eta}_{\cong} & E_0 \ar[r]^{h} & B^{\times I}
}\]
sending $1_B$ to $h \eta(1_B) = (\eta^{-1}(\eta(1_B)))_{1\times I} = (1_B)_{1\times I}$. Thus, we have a composition $\tilde{\theta}_0: = \theta h \eta : B \to N$, which is an $(A,B)$-homomorphism satisfying $\tilde{\theta}_0 (1_B) = v$.
Now, for each $u\in \NN$, we define $\theta_u$ as follows:
\begin{itemize}
  \item[(1)]
    $\theta_0: E_0 \to N$ is defined as $\theta_0 := \tilde{\theta}_0\eta^{-1} = \theta h$.
    Here, the element $\eta(1_B)$ in $E_0$ is written as $1$ (in this notation, we have $\theta h(1) = \theta h \eta(1_B)= \theta((1_B)_{1\times I}) = v$) and, up to isomorphism, we do not differential between $B$ and $E_0$ for simplicity.
  \item[(2)]
    $\theta_{u+1}$ is induced by $\theta_u$ through the composition
    \[ \theta_{u+1}:= \delta \compos \theta_u^{\oplus_p 2^{\dimA}}\compos \gamma_{\xi}|_{E_{u+1}}^{-1}  \ :
    \xymatrix@C=1.5cm{
      E_{u+1} \ar[r]^{\gamma_{\xi}|_{E_{u+1}}^{-1}}_{ {\cong \atop \text{see Lemma \ref{lemm:Eu}}} }
    & E_u^{\oplus_p 2^{\dimA}} \ar[r]^{\theta_u^{\oplus_p 2^{\dimA}}}
    & N^{\oplus_p 2^{\dimA}} \ar[r]^{\delta}
    & N
    }. \]
\end{itemize}
Then
\begin{align*}
\theta_{u+1}(\mathbf{1}_{\II_A}|_{E_{u+1}})
& = \delta \compos \theta_{u}^{\oplus_p 2^{\dimA}}  \compos \gamma_{\xi}|_{E_{u+1}}^{-1} (\mathbf{1}_{\II_A}|_{E_{u+1}})
  = \delta( \theta_{u}^{\oplus_p 2^{\dimA}} (\mathbf{1}_{\II_A}|_{E_u})_{1\times 2^{\dimA}})  \\
& = \delta( (\theta_{u}(\mathbf{1}_{\II_A}|_{E_u}))_{1\times 2^{\dimA}})
\end{align*}
In the case of $n=1$, the above equation admits
\begin{align*}
 \theta_{1}(\mathbf{1}_{\II_A}|_{E_{1}})
& = \delta( (\theta_{0}(\mathbf{1}_{\II_A}|_{E_0}))_{1\times 2^{\dimA}})
  = \delta( (\theta_{0}\eta(1_B))_{1\times 2^{\dimA}}) \\
& = \delta( (\theta_{0}(1))_{1\times 2^{\dimA}})
  = \delta( (\theta h(1))_{1\times 2^{\dimA}})
  = \delta( (v)_{1\times 2^{\dimA}}) = v,
\end{align*}
and, for any $k\in\NN$ with $\theta_{k}(\mathbf{1}_{\II_A}|_{E_{k}}) = v$, we have
\begin{align*}
  \theta_{k+1}(\mathbf{1}_{\II_A}|_{E_{k+1}})
& = \delta( (\theta_{k}(\mathbf{1}_{\II_A}|_{E_k}))_{1\times 2^{\dimA}}) \\
& = \delta( (v)_{1\times 2^{\dimA}}) = v.
\end{align*}
Therefore, we have
\begin{align}\label{formula: lemm-existence}
  \theta_{u+1}(\mathbf{1}_{\II_A}|_{E_{u+1}}) = v
\end{align}
for all $u\in\NN$ by induction.

Consider the maps $\alpha_i : E_i \to \dirlim E_t$ and $\alpha_{ij}: E_i \to E_j$ ($i,j\in\NN$ and $i\=< j$) induced by $E_i \subseteq E_j \subseteq \dirlim E_t$ (see Lemmas \ref{lemm:Eu subset bfS}), we have that the diagram shown in \Pic \ref{fig:existence} commutes,
where $\theta_{\lim}: \dirlim E_t \to N$ is given by the inductive limit $\dirlim E_t$ of the direct system $((E_u)_{u\in \NN}, (\alpha_{uv})_{u\=< v})$.
\begin{figure}[htbp]
  \centering
\begin{tikzpicture}
\draw[rotate=15+  0][<-][line width=1pt] (2,0) arc(0:90:2) [red][dashed];
\draw[rotate=15+120][<-][line width=1pt] (2,0) arc(0:90:2);
\draw[rotate=15+240][->][line width=1pt] (2,0) arc(0:90:2);
\draw (2,0) node{$N$};
\draw[rotate=120] (2,0) node{$\dirlim E_t$};
\draw[rotate=240] (2,0) node{$E_{u+1}$};
\draw (0,0) node{$E_u$};
\draw[->][line width=1pt] (0.5,0) to (1.5,0);
\draw[rotate=120][->][line width=1pt] (0.5,0) to (1.5,0);
\draw[rotate=240][->][line width=1pt] (0.5,0) to (1.5,0);
\draw[rotate=60+  0] (2.4,0) node[red]{$\exists~\theta_{\lim}$};
\draw[rotate=60+120] (2.4,0) node{$\alpha_{u+1}$};
\draw[rotate=60+240] (2.4,0) node{$\theta_{u+1}$};
\draw[rotate= 0+  0] (1.0,0.3) node{$\theta_u$};
\draw[rotate= 0+120] (1.0,0.3) node{$\alpha_u$};
\draw[rotate= 0+120] (1.0,-0.3) node{$\subseteq$};
\draw[rotate= 0+255] (1.0,0.3) node{$\alpha_{u~u+1}$};
\draw[rotate= 0+235] (1.0,-0.3) node{{$\subseteq$}};
\end{tikzpicture}
\caption{The existence of $(A,B)$-homomorphism $(\w{\bfS_{\homo}(\II_A)}, \mathbf{1}_{\II_A}, \w{\gamma}_{\xi}) \to (N,v,\delta)$.}
\label{fig:existence}
\end{figure}
By Lemma \ref{lemm:limits}, we have $\rho: \w{\bfS_{\homo}(\II_A)} \mathop{\to}\limits^{\cong} \dirlim E_t $. Thus, we obtain an $(A,B)$-homomorphism $\tilde{\theta} := \theta_{\lim}\compos\rho: \w{\bfS_{\homo}(\II_A)} \to N$.

We need show that $\tilde{\theta}$ is a morphism in $\scrN_{\homo}^p$.
On the one hand, up to the isomorphism $\rho$, \ref{H1} holds since the following formulas
\[ \tilde{\theta}(\mathbf{1}_{\II_A})
= \invlim~\theta_{\lim}|_{E_t}(\mathbf{1}_{\II_A}|_{E_t})
= \invlim~\theta_{\lim}|_{E_t}(\alpha_t(\mathbf{1}_{\II_A}|_{E_t}))
= \invlim~\theta_t(\mathbf{1}_{\II_A}|_{E_t})
\mathop{=\!=}\limits^{(\ref{formula: lemm-existence})}
  \invlim v = v. \]
On the other hand, let $E_u^{\oplus} := E_u^{\oplus_p 2^{\dimA}}$ and $N^{\oplus} := N^{\oplus_p 2^{\dimA}}$.
For each $\pmb{f}=(f_1,\ldots, f_{2^{\dimA}})\in \w{\bfS}$,
it can be seen as the projective limit $\underleftarrow{\lim}\pmb{f}_i$ of a sequence $\{\pmb{f}_i=(f_{1i},\ldots, f_{2^{\dimA}i})\}_{i\in\NN}$ in $\bigcup_{u\in\NN} E_u^{\oplus}$,
where $f_{ji} \in E_{u_i}$ ($1\=< j\=< 2^{\dimA}$), $u_i \in \NN$, such that for any $i\=< j$, we have $u_i \=< u_j$.
Thus, naturally, we need to consider the following diagram up to the isomorphism $\rho$:
\[\xymatrix@C=2cm@R=2cm{
  E_{u_i}^{\oplus}
  \ar[r]^{\gamma_{\xi}|_{E_{u_i}^{\oplus}}}_{\cong}
  \ar@{^(->}[d]_{e_{u_i}^{\oplus 2^{\dimA}}}
  \ar@/_6pc/[dd]_{\theta_{u_i}^{\oplus 2^{\dimA}}}
& E_{u_i+1}
  \ar@{^(->}[d]^{e_{u_i+1}}
  \ar@/^6pc/[dd]^{\theta_{u_i}}
\\
  \w{\bfS}^{\oplus}
  \ar[r]^{\w{\gamma}_{\xi}}
  \ar[d]_{\tilde{\theta}^{\oplus 2^{\dimA}}}
& \w{\bfS_{\homo}(\II_{A})}
  \ar[d]^{\tilde{\theta}}
\\
  N^{\oplus}
  \ar[r]_{\delta}
& N,
}\]
where, for each $t\in \NN$, $e_t := \rho\alpha_t$ is an embedding. We have the following equation
\begin{align}
    \tilde{\theta}(\w{\gamma}_{\xi}(\pmb{f}))
&  = \invlim~\tilde{\theta}(\w{\gamma}_{\xi}( e^{\oplus 2^{\dimA}}_{u_i}(\pmb{f}_i)))
    && \nonumber \\
& = \invlim~\tilde{\theta}(e_{u_i+1}(\gamma_{\xi}\big|_{E^{\oplus_p 2^{\dimA}}}(\pmb{f}_i)))
    &&
        ( 
        \w{\gamma}_{\xi} e^{\oplus 2^{\dimA}}_{u_i} = e_{u_i+1} \gamma_{\xi}\big|_{E^{\oplus}} ) \nonumber \\
& = \invlim~\theta_{u_i}(\gamma_{\xi}\big|_{E^{\oplus}}(\pmb{f}_i))
    &&
        (
        \tilde{\theta} e_{u_i+1} =\theta_{u_i}) \nonumber \\
& = \invlim~\delta(\theta_{u_i}^{\oplus 2^{\dimA}}(\pmb{f}_i))
    &&
        (
        \theta_{u_i} \gamma_{\xi}\big|_{E^{\oplus} } = \delta\theta_{u_i}^{\oplus 2^{\dimA}}) \nonumber \\
& = \invlim~\delta(\tilde{\theta}^{\oplus 2^{\dimA}}(e^{\oplus 2^{\dimA}}_{u_i}(\pmb{f}_i)))
    &&
        (
        \theta_u^{\oplus 2^{\dimA}} = \tilde{\theta}^{\oplus 2^{\dimA}}e^{\oplus 2^{\dimA}}_{u_i}) \nonumber \\
& = \delta(\invlim~\tilde{\theta}^{\oplus 2^{\dimA}}(e^{\oplus 2^{\dimA}}_{u_i+1}(\pmb{f}_i)))
    && \text{\ref{Np3}}. \label{formula: lemm-existence 250701}
\end{align}
Notice that the definition of $\{E_u\}_{u\in\NN}$ provide a disjoint union $\II_A = \bigcup\limits_{i=1}^{\dimA} \II_i$ of $\II_A$, this union admits that each function $g$ in $E_{u+1}$ is a sequence $(g_j: \II_j \to B)_{1\=< j \=< 2^{\dimA}}$ which can be seen as an element lying in $E_{u-1}^{\oplus}$, and then, for the case of $u=1$, $g$ is of the form
\[ g = \sum_{j=1}^{2^{\dimA}} b_i \mathbf{1}_{\II_j}. \]
Thus, up to the isomorphism $\eta: B \mathop{\to}\limits^{\cong} E_0$ given in Lemma \ref{lemm:E0}, one can check that the norm of $\theta_1$ is
\begin{align*}
  \Vert \theta_1 \Vert
& = \sup_{\Vert g\Vert_{E_1} =1} \Vert\theta_1(g)\Vert_N \\
& = \sup_{\sum_{i=1}^{d_A} b_i\mu(\II_i)=1} \Vert \delta ((\theta_0(b_i))_{1\times 2^{\dimA}}) \Vert_N
  = \Vert \delta \Vert,
\end{align*}
and then one can prove that $\Vert\theta_t\Vert = \Vert\delta\Vert$ holds for all $t\in \NN$ by induction, and so, $\Vert \tilde{\theta} \Vert = \Vert\theta_{\lim}\Vert = \Vert\delta\Vert$ holds since $\theta_{\lim}$ is given by the projective limit of the direct system $((E_u)_{u\in \NN}, (\alpha_{uv})_{u\=< v})$.
Thus, $\tilde{\theta} = \theta\compos\rho$ is a bounded $\FF$-linear map, and so is $\tilde{\theta}^{\oplus 2^{\dimA}}$.
Then \[\invlim~\tilde{\theta}^{\oplus 2^{\dimA}}(e^{\oplus 2^{\dimA}}_{u_i+1}(\pmb{f}_i))
 = \tilde{\theta}^{\oplus 2^{\dimA}}(\invlim e^{\oplus 2^{\dimA}}_{u_i+1}(\pmb{f}_i))
 = \tilde{\theta}^{\oplus 2^{\dimA}}(\invlim \pmb{f}_i)
 = \tilde{\theta}^{\oplus 2^{\dimA}}(\pmb{f}).\]
It follows that (\ref{formula: lemm-existence 250701}) admits $\tilde{\theta}\w{\gamma}_{\xi} = \delta\tilde{\theta}^{\oplus 2^{\dimA}}$, i.e., \ref{H2} holds.
\end{proof}

\begin{proposition} \label{prop:uniq}
For any object $(N,v,\delta)$ in $\scrA^p$, if $\Hom_{\scrA^p_{\homo}}(
  (\w{\bfS_{\homo}(\II_A)}, \mathbf{1}_{\II_A}, \w{\gamma}_{\xi}),
  (N,v,\delta)$ contains at least one morphism, then
\[ \sharp\Hom_{\scrA^p_{\homo}}(
  (\w{\bfS_{\homo}(\II_A)}, \mathbf{1}_{\II_A}, \w{\gamma}_{\xi}),
  (N,v,\delta)
  ) = 1. \]
\end{proposition}

\begin{proof}
\def\itLamb{A}
\def\id{\mathrm{id}}
Keep the notation from the proof of Proposition \ref{prop:existence}.
Assume $\Hom_{\scrA^p_{\homo}}(
  (\w{\bfS_{\homo}(\II_A)}$, $\mathbf{1}_{\II_A}, \w{\gamma}_{\xi}),
  (N,v,\delta)
  )$ contains two morphism $h$ and $h'$. Then the square
\[\xymatrix@C=1.5cm{
  E_u^{\oplus}
  \ar[r]^{\gamma_{\xi}|_{E_u^{\oplus}}}_{\cong}
  \ar[d]_{(h|_{E_u}-h'|_{E_u})^{\oplus 2^{\dimA}}}
& E_{u+1}
  \ar[d]^{h|_{E_{u+1}}-h'|_{E_{u+1}}} \\
  N^{\oplus}
  \ar[r]_{\delta}
& N
}\]
commutes for all $u\in \NN$, and then for any $f\in E_{u+1}$, we have
\[ (h|_{E_{u+1}}-h'|_{E_{u+1}})(f)
= (\delta\compos (h|_{E_u}-h'|_{E_u})^{\oplus 2^{\dimA}}\compos(\gamma_{\xi}|_{E_u^{\oplus}})^{-1})(f).\]
Thus, $h|_{E_{u+1}}-h'|_{E_{u+1}}$ is determined by $h|_{E_u}-h'|_{E_u}$.
If $u=0$, then \[(h|_{E_0}-h'|_{E_0})(k\mathbf{1}_{\II_A}|_{E_0})
= k(h|_{E_0}(\mathbf{1}_{\II_A}|_{E_0})-h'|_{E_0}(\mathbf{1}_{\II_A}|_{E_0})) = k(v-v)=0, \]
it follows $h|_{E_0}=h'|_{E_0}$. Therefore, one can prove that $h|_{E_u}=h'|_{E_u}$ for all $u\in\NN$ by induction.

The direct system $\big((E_i)_{i\in\NN}, (e_{ij} :E_i \mathop{\to}\limits^{\subseteq} E_j)_{i\=< j}\big)$ provides a commutative diagram shown in \Pic \ref{fig:uniq} for all $i \=< j$,
\begin{figure}[htbp]
  \centering
\begin{tikzpicture}[scale=1.0]
\draw[rotate=15+  0][<-][line width=1pt] (2,0) arc(0:90:2) [red][dashed];
\draw[rotate=15+120][<-][line width=1pt] (2,0) arc(0:90:2);
\draw[rotate=15+240][->][line width=1pt] (2,0) arc(0:90:2);
\draw (2,0) node{$N$};
\draw[rotate=120] (2,0) node{$\w{\bfS_{\homo}(\II_A)}$};
\draw[rotate=240] (2,0) node{$E_{j}$};
\draw (0,0) node{$E_i$};
\draw[->][line width=1pt] (0.5,0) to (1.5,0);
\draw[rotate=120][->][line width=1pt] (0.5,0) to (1.5,0);
\draw[rotate=240][->][line width=1pt] (0.5,0) to (1.5,0);
\draw[rotate=60+  0] (2.4,0) node[red]{$\phi$};
\draw[rotate=60+120] (2.4,0) node{$e_j$};
\draw[rotate=60+120] (2.0,0) node[right]{$\subseteq$};
\draw[rotate=60+250] (2.4,0) node[right]{${}_{h|_{E_j}-h'|_{E_j}~(=0)}$};
\draw[rotate= 0+  0] (1.0,0.3) node{\tiny$h|_{E_i}-h'|_{E_i}$};
\draw[rotate= 0+  0] (1.0,-0.3) node{\tiny$(=0)$};
\draw[rotate= 0+120] (1.0,0.3) node{$e_i$};
\draw[rotate= 0+120] (1.0,-0.3) node{$\subseteq$};
\draw[rotate= 0+255] (1.0,0.3) node{$e_{ij}$};
\draw[rotate= 0+235] (1.0,-0.3) node{{$\subseteq$}};
\end{tikzpicture}
\caption{The uniqueness of $(A,B)$-homomorphism $(\w{\bfS_{\homo}(\II_A)}, \mathbf{1}_{\II_A}, \w{\gamma}_{\xi}) \to (N,v,\delta)$.}
\label{fig:uniq}
\end{figure}
where $\phi: \w{\bfS_{\homo}(\II_A)} \to N$ is obtained by $\dirlim E_i \cong \w{\bfS_{\homo}(\II_A)}$.
Since $(h-h')\compos e_{ij} = h|_{E_i}-h'|_{E_j}$, we know that the case for $\phi=h-h'$ makes the above diagram commute. Moreover, the case for $\phi=0$ makes the above diagram commute. Thus, we obtain $h-h'=0$ and $h=h'$.
\end{proof}

By Propositions \ref{prop:existence} and \ref{prop:uniq}, we obtain the following result, which is the first main result of this paper.

\begin{theorem} \label{thm:1}
The triple $(\w{\bfS_{\homo}(\II_A)}, \mathbf{1}_{\II_A}, \w{\gamma}_{\xi})$ is an initial object in $\scrA^p_{\homo}$.
\end{theorem}

\subsection{\texorpdfstring{Special objects in $\scrN^p_{\homo}$}{Special objects in Np} } \label{subsect:sp.obj}

Now we consider some special objects in $\scrN^p_{\homo}$.

\subsubsection{\texorpdfstring{$\scrN^p_{\homo}$-initial objects}{Np-initial object} }

We recall some concepts in \cite[Chapter 5]{R1979}. Let $\mathcal{C}$ be a category. An object $O$ in $\mathcal{C}$ is called {\defines initial} if it holds for any object $Y$ that $\Hom_{\mathcal{C}}(O, Y)$ contains only one morphism. The initial object in $\mathcal{C}$ is unique up to isomorphism. Let $\mathcal{D}$ be a full subcategory of $\mathcal{C}$.
In \cite{LLHZ2025}, authors introduced $\mathcal{D}$-initial object which is a generalization of initial object.

\begin{definition} \rm
An object $C \in \mathcal{C}$ is called {\defines $\mathcal{D}$-initial} if for any $D\in\mathcal{D}$, there is a unique morphism $h:C\to D$ such that
\[
\xymatrix@C=1.5cm{
 C \ar[r]^{h} \ar[d]_{\subseteq} & D \\
 D' \ar[ru]_{h'} &
}
\]
commutes, where $D'$ is an initial object in $\mathcal{D}$ and $h'$ is a morphism in $\mathcal{D}$.
\end{definition}

For any object $C'$ that is a subobject of $D'$ in $\mathcal{C}$, consider any morphism $\hbar : C'\to D$ such that the following diagram
\[
\xymatrix@C=1.5cm{
 C' \ar[r]^{\hbar} \ar[d]_{e}^{\subseteq} & D \\
 D' \ar[ru]_{h'} &
}
\]
commutes. We have always $\hbar=h'e$. By the uniqueness of $h'$, $\hbar$ is unique. Then we immediately obtain the following lemma.

\begin{lemma} \label{lemm:initial}
Let $\mathcal{C}$ be a category and $\mathcal{D}$ a subcategory of $\mathcal{C}$, and let $D'$ be an initial object in $\mathcal{D}$.
If an object $C$ is a subobject of $D'$ in $\mathcal{C}$, then $C$ is a $\mathcal{D}$-initial object.
\end{lemma}

The following result is the second main result of this paper.

\begin{theorem} \label{thm:2}
The triple $(\bfS_{\homo}(\II_A), \mathbf{1}_{\II_A}, \gamma_{\xi})$ in $\scrN^p_{\homo}$ is an $\scrA^p_{\homo}$-initial object. To be more precise, for any object $(N, v,\delta)$, there is a unique morphism $h:(\bfS_{\homo}(\II_A),\mathbf{1}_{\II_A},\gamma_{\xi}) \to (N,v,\delta)$ in $\scrN^p_{\homo}$, such that the diagram
\[ \xymatrix@C=1.5cm{
  (\bfS_{\homo}(\II_A),\mathbf{1}_{\II_A},\gamma_{\xi})
  \ar[r]^{h} \ar[d]_{\subseteq}
& (N,v,\delta) \\
 (\w{\bfS_{\homo}(\II_A)},\mathbf{1}_{\II_A},\w{\gamma}_{\xi})
  \ar[ru]_{\w{h}}
 &
} \]
commutes. Here, $\w{h}$ is an $(A,B)$-homomorphism induced by the completion $\w{\bfS_{\homo}(\II_A)}$ of $\bfS_{\homo}(\II_A)$, and it is an extension of $h$.
\end{theorem}

\begin{proof}
Since $\w{\bfS_{\homo}(\II_A)}$ is a completion of $\bfS_{\homo}(\II_A)$, we have an embedding $\bfS_{\homo}(\II_A) \mathop{\to}\limits^{\subseteq} \w{\bfS_{\homo}(\II_A)}$,
it follows that $(\bfS_{\homo}(\II_A), \mathbf{1}_{\II_A}, \gamma_{\xi})$ is a subobject of $(\w{\bfS_{\homo}(\II_A)}, \mathbf{1}_{\II_A}, \w{\gamma}_{\xi})$.
Then by Lemma \ref{lemm:initial}, we obtain that $(\bfS_{\homo}(\II_A), \mathbf{1}_{\II_A}, \gamma_{\xi})$, as an object in $\scrN^p_{\homo}$, is an $\scrA^p_{\homo}$-initial object since $(\w{\bfS_{\homo}(\II_A)}, \mathbf{1}_{\II_A}, \w{\gamma}_{\xi})$ is an initial object in $\scrA^p_{\homo}$ (see Theorem \ref{thm:1}).
\end{proof}

\subsubsection{An important object}
Let $\FF$, $A$ and $B$ are completed, i.e., $\w{\FF}=\FF$, $\w{A}=A$ and $\w{B}=B$. In this subsection we provide another object in $\scrA^p_{\homo}$. Recall that under the action of $\kappa_c$ and $\kappa_d$, $\II_A$ is divided to $2^{\dimA}$ subsets which are of the form
\[ \Pi_{(\sigma_1,\ldots,\sigma_{\dimA})} = \kappa_{\sigma_1}([c,d]_{\FF}) \times \kappa_{\sigma_2}([c,d]_{\FF})
  \times \cdots \times \kappa_{\sigma_{\dimA}}([c,d]_{\FF}), \]
where $(\sigma_1,\ldots,\sigma_{\dimA}) \in \{c,d\}^{\times \dimA}$, see \ref{subsubsect:Eu}. For simplicity, we define $\{\II_i \mid 1\=< i \=< 2^{\dimA}\}$ is the set of all $\Pi_{(\sigma_1,\ldots,\sigma_{\dimA})}$ as above.


\begin{lemma} \label{lemm:A}
Let $\ave : B^{\oplus_p 2^{\dimA}} \to B$ be the map defined as
\begin{align*}
  (b_1,b_2,\ldots, b_{2^{\dimA}})
& := (b_{
          (\sigma_1,\ldots,\sigma_{2^{\dimA}})
        })|_{
             (\sigma_1,\ldots,\sigma_{2^{\dimA}})
             \in \{c,d\}^{\times \dimA}
            } \\
& \mapsto
   \sum_{
          \begin{smallmatrix}
             (\sigma_1,\ldots,\sigma_{2^{\dimA}}) \\
             \in \{c,d\}^{\times \dimA}
          \end{smallmatrix}
        }
        \frac
          {\mu_{\II_A}(\Pi_{(\sigma_1,\ldots,\sigma_{\dimA})})}
          {\mu_{\II_A}(\II_A)} b_{(\sigma_1,\ldots,\sigma_{\dimA})} \\
& =: \hspace{0.6cm}
    \sum_{i=1}^{2^{\dimA}} \frac{\mu_{\II_A}(\II_i)}{\mu_{\II_A}(\II_A)} b_i.
\end{align*}
If $\FF$ is an extension of $\RR$, then $\ave$ is an $(A,B)$-homomorphism sending $(\mu_{\II_A}(\II_A)1_B)_{1\times 2^{\dimA}}$ to $\mu_{\II_A}(\II_A)1_B$.
\end{lemma}

\begin{proof}
By the definition of $\ave$, we have
\[ \ave((1_B)_{1\times 2^{\dimA}})
= \sum_{i=1}^{2^{\dimA}} \frac{\mu_{\II_A}(\II_i)}{\mu_{\II_A}(\II_A)} 1_B
= \frac{1_B}{\mu_{\II_A}(\II_A)} \sum_{i=1}^{2^{\dimA}} \mu_{\II_A}(\II_i)
= 1_B, \]
where $\displaystyle \sum_{i=1}^{2^{\dimA}} \mu_{\II_A}(\II_i) = \mu_{\II_A}(\II_A)$ holds since $\mu_{\II_A}$ is a measure.
Next, we prove that $\ave$ is an $(A,B)$-homomorphism. The proof of $\ave$ being an $\FF$-linear map is left for readers.
We need prove that $\ave(a.(b_1,\ldots,b_{2^{\dimA}}).b) = a.\ave((b_1,\ldots,b_{2^{\dimA}})).b$ holds for all $a\in A$ and $b\in B$.
By the definition of $\ave$, we have
\begin{align} \label{formula: lemm-A}
   \ave(a.(b_1,\ldots,b_{2^{\dimA}}).b)
=  \ave((\homo(a)b_1b,\ldots,\homo(a)b_{2^{\dimA}}b))
=  \sum_{i=1}^{2^{\dimA}}
      \frac{\mu_{\II_A}(\II_i)}{\mu_{\II_A}(\II_A)}
        \homo(a)b_ib.
\end{align}
Notice that each $\homo(a)b_i$, as an element in $B=\alg(\Q_B,\vecd_b,\vecg_B,\EE, (\FF_i)_{i\in\Q_0}))$, is a finite sum
\[ \homo(a)b_i = \sum_{\wp=a_{\wp,1}\cdots a_{\wp,\ell} \in (\Q_B)_{\>=0}} k_{i,\wp}, \]
where
\[ k_{i,\wp} \in A_{\wp}
= \bigotimes_{j=1}^{\ell}
     \FF_{\fcts(a_{\wp,j})}
       \otimes_{
         \FF_{\fcts(a_{\wp,j})} \cap \FF_{\fctt(a_{\wp,j})}
         }
        \FF_{\fctt(a_{\wp,j})}^{g_{ a_{\wp,j} }}, \]
and $g_{a_{\wp,j}}$ is an $\FF$-automorphism in Galois group $\Gal(\FF_{\fcts(a_{\wp,j})} \cap \FF_{\fctt(a_{\wp,j})}/\FF)$, then we have
\[ \homo(a)b_i \frac{\mu_{\II_A}(\II_i)}{\mu_{\II_A}(\II_A)}
 = g_{a_{\wp,1}} \compos\cdots \compos g_{a_{\wp,\ell}}
     \Big(
       \frac{\mu_{\II_A}(\II_i)}{\mu_{\II_A}(\II_A)}
     \Big)\homo(a)b_i
 = \homo(a)b_i \frac{\mu_{\II_A}(\II_i)}{\mu_{\II_A}(\II_A)}\]
by $\displaystyle \frac{\mu_{\II_A}(\II_i)}{\mu_{\II_A}(\II_A)} \in \RR \subseteq \FF$.
Thus, (\ref{formula: lemm-A}) yields
\begin{align*}
    \ave(a.(b_1,\ldots,b_{2^{\dimA}}).b)
&=  \sum_{i=1}^{2^{\dimA}}
    \homo(a) \frac{\mu_{\II_A}(\II_i)}{\mu_{\II_A}(\II_A)} b_ib
 =  \homo(a)\bigg(\sum_{i=1}^{2^{\dimA}}
    \frac{\mu_{\II_A}(\II_i)}{\mu_{\II_A}(\II_A)} b_i\bigg) b \\
&= \homo(a)\ave((b_1,\ldots,b_{2^{\dimA}}))b
 = a.\ave((b_1,\ldots,b_{2^{\dimA}})).b,
\end{align*}
Then $\ave$ is an $(A,B)$-homomorphism. Furthermore, the following formula
\[ \ave((\mu_{\II_A}(\II_A)1_B)_{1\times 2^{\dimA}})
= \mu_{\II_A}(\II_A) \ave((1_B)_{1\times 2^{\dimA}})
= \mu_{\II_A}(\II_A) 1_B \]
holds by using this fact.
\end{proof}

\begin{lemma} \label{lemm:A limit}
In the case of $\FF$, $A$, and $B$ being completed, we have $\ave(\invlim x_t) = \invlim \ave(x_t)$ for any Cauchy sequence $\{x_t\}_{t\in\NN}$ in $B$.
\end{lemma}

\begin{proof}
Since $B$, as an $\FF$-vector space, is finite-dimensional, then so is $B^{\oplus_p 2^{\dimA}}$. It is well-known that any linear map defined on a finite-dimensional vector space is continuous, then $\ave$ is continuous since all $(A,B)$-homomorphisms are $\FF$-linear. Thus, $\ave(\invlim x_t) = \invlim \ave(x_t)$ holds for all Cauchy sequence $\{x_t\}_{t\in\NN}$.
\end{proof}

\begin{proposition} \label{prop:A}
The triple $(B, \mu_{\II_A}(\II_A), \ave)$ is an object in $\scrA^p_{\homo}$.
\end{proposition}

\begin{proof}
Since $B$ with the map (\ref{norm B}) is a normed $(A,B)$-module, \ref{Np1} holds. Lemmas \ref{lemm:A} and \ref{lemm:A limit} provides \ref{Np2} and \ref{Np3}. Thus, $(B, \mu_{\II_A}(\II_A), \ave)$ is an object in $\scrN^p_{\homo}$. Moreover, $\FF$, $A$, and $B$ are complete, it follows that $(B, \mu_{\II_A}(\II_A), \ave)$ is an object in $\scrA^p_{\homo}$.
\end{proof}

\section{Applications I: Abstract integrations} \label{sect:app-int}

\def\bfF{\mathbf{F}}
\def\INT{\mathfrak{I}}

Abstract integral is a general form of Reimann/Lebesgue integral, which was first introduced by Daniell in \cite[Page 280]{Dan1918}. Moreover, Daniell considered other generalizations of integrations, such as \cite{Dan1919-I, Dan1919-II, Dan1921}. Nowadays, there are multiple versions of the definition of abstract integral, and some literature also provides axiomatic versions of the definition of Daniell integral, cf. \cite[etc]{Royden1988}.

\subsection{Daniell integrations} \label{subsect:Den int}
We recall the original definition of Daniell integrals in the next paragraph.

Let $\bfF(X)$ be a family of bounded real functions defined over a set $X$ such that the following two conditions hold:
\begin{itemize}
  \item[(1)] $\bfF(X)$ is an $\RR$-vector space;
  \item[(2)] if $f\in \bfF(X)$, then $|f|: X \to \RR$, $x\mapsto |f(x)|$ lies in $\bfF(X)$.
\end{itemize}
The {\defines Daniell integral} of a function $h\in F$ is the image $\INT(h)$ of $h$ given by the map $ \INT : \bfF(X) \to \RR$, where $\INT$ satisfies the following conditions.
\begin{enumerate}[label=(D\arabic*)]
  \item for arbitrary $h_1,h_2 \in \bfF(X)$, $k_1,k_2\in\RR$:
    $\INT(k_1h_1+k_2h_2) = k_1\INT(h_1) + k_2 \INT(h_2)$;
    \label{D1}
  \item for each $h\in \bfF(X)$ with $\Ima(h) \in \RR^{\>=0}$, we have $\INT(h)\>= 0$;
    \label{D2}
  \item  for each nonincreasing sequence $\{h_t\}_{t\in\NN^+}$, if $\displaystyle \lim_{t\to+\infty} h_t(x) = 0$ holds for all $x\in X$, then $\displaystyle \lim_{t\to+\infty} \INT(h_t) = 0$. \label{D3}
\end{enumerate}

If we want to consider the abstract integral of a function $f: \II_A \to B$ in $\w{\bfS_{\homo}(\II_A)}$, since $B$ may not necessarily have a partial order, the conditions \ref{D1}, \ref{D2}, and \ref{D3} need to be modified by the following.
\begin{enumerate}[label=($\mathfrak{I}$\arabic*)]
  \item $\INT$ is an $(A,B)$-homomorphism;
    \label{I1}
  \item for each $h\in \bfF(X)$, we have $\INT(\Vert h\Vert_{B,p} 1_B) = \omega1_B \in \RR^{\>=0}1_B$;
    \label{I2}
  \item for each  nonincreasing Cauchy sequence $\{h_t\}_{t\in\NN^+}$ in $\w{\bfS_{\homo}(\II_A)}$ with $\invlim h_t = 0$, we have $\invlim \INT(h_t) = 0$. \label{I3}
\end{enumerate}

\begin{theorem} \label{thm:3}
Assume that $p=1$, $\FF$ is an extension of $\RR$, and $A$ and $B$ are completed. Then there exists a unique morphism $T: (\bfS_{\homo}(\II_A), \mathbf{1}_{\II_A}, \gamma_{\xi}) \to (B, \mu_{\II_A}(\II_A)1_B, \ave)$ in ${\scrN^1_{\homo}}$ such that
\[ \xymatrix@C=1.5cm{
  (\bfS_{\homo}(\II_A),\mathbf{1}_{\II_A},\gamma_{\xi})
  \ar[r]^{T} \ar[d]_{\subseteq}
& (B, \mu_{\II_A}(\II_A)1_B, \ave) \\
  (\w{\bfS_{\homo}(\II_A)},\mathbf{1}_{\II_A},\w{\gamma}_{\xi})
  \ar[ru]_{\w{T}}
 &
} \]
commutes. Here, $\w{T}$ is an $(A,B)$-homomorphism in $\scrA^p_{\homo}$ induced by the completion $\w{\bfS_{\homo}(\II_A)}$ of $\bfS_{\homo}(\II_A)$. Furthermore,
\begin{itemize}
  \item[\rm(1)] $\w{T}$ sends each function $f = \sum_i b_i \mathbf{1}_{I_i} \in \bfS_{\homo}(\II_A)$ $(\forall i\ne j, I_i \cap I_j = \varnothing$, and $\II_A = \bigcup_i I_i)$ to an element $\sum_i b_i \mu_{\II_A}(I_i)$;
  \item[\rm(2)] and $\w{T}$ satisfies {\rm\ref{I1}}, {\rm\ref{I2}}, and {\rm\ref{I3}}.
\end{itemize}
\end{theorem}

\begin{proof} (1)
Every function $f$ in $\bfS_{\homo}(\II_A)$ is of the form $\displaystyle f = \sum\nolimits_{i} b_i \mathbf{1}_{I_i}$. Consider the map
\[ \tilde{T}: \bfS_{\homo}(\II_A) \to B, ~ f\mapsto \sum\nolimits_{i} b_i \mu_{\II_A}(I_i). \]
We need show that $\tilde{T} \in \Hom_{\scrN^1_A}((\bfS_{\homo}(\II_A), \mathbf{1}_{\II_A}, \gamma_{\xi}), (B, \mu_{\II_A}(\II_A)1_B, \ave))$.
First of all, for all $a\in A$, $b\in B$, we have
\begin{align}
  \tilde{T}(a.f.b)
& = \tilde{T}\Big(\homo(a) \Big(\sum\nolimits_{i} b_i \mathbf{1}_{I_i}\Big) b\Big)
  = \tilde{T}\Big(\sum\nolimits_{i} \homo(a) b_i \mathbf{1}_{I_i}b\Big) \nonumber \\
& \mathop{=}\limits^{\spadesuit}
    \tilde{T}\Big(\sum\nolimits_{i} \homo(a) b_i b\mathbf{1}_{I_i}\Big)
  = \sum\nolimits_{i} \homo(a) b_i b \mu_{\II_A}(I_i) \label{formula: thm-3}
\end{align}
where, $\spadesuit$ is given by $\mathbf{1}_{I_i}b=b\mathbf{1}_{I_i}$, which can be proved by using the definition of $\mathbf{1}_{I_i}$ and two trivial facts $1_Bb=b=b1_B$ and $0_Bb=0_B=b0_B$.
Recall that $B$ is the tensor ring $\alg(\Q_B,\vecd_B,\vecg_B,\FF,(\FF_i)_{i\in(\Q_B)_0})$ (see Subsection \ref{sect:norm tens ring}), there is a family of elements $\{k_{\wp} \in A_{\wp} \mid \wp \in (\Q_B)_{\>= 0}\}$ such that
\[ b = \sum_{\wp =a_{\wp,1}a_{\wp,2}\cdots a_{\wp,\ell} \in (\Q_B)_{\>= 0}} k_{\wp}, \]
where
\[ A_{\wp}
     = \bigotimes_{j=1}^{\ell}
          \FF_{\fcts(a_{\wp,j})}
          \otimes_{
              \FF_{\fcts(a_{\wp,j})} \cap \FF_{\fctt(a_{\wp,j})}
            }
          \FF_{\fctt(a_{\wp,j})}^{g_{ a_{\wp,j} }},
\]
and each $g_{ a_{\wp,j} }$ is an $\FF$-automorphism in the Galois group $\Gal(\FF_{\fcts(a_{\wp,j})} \cap \FF_{\fctt(a_{\wp,j})} / \FF)$. Thus, we have $k_{\wp} \mu_{\II_A}(I_i) = g_{ a_{\wp,1} } \compos g_{ a_{\wp,2} } \compos \cdots \compos g_{ a_{\wp,t} } (\mu_{\II_A}(I_i)) k_{\wp}$.
Since $\FF$ is an extension of $\RR$ and $\mu_{\II_A}(I_i)\in \RR$ is also an element in $\FF$, we obtain
$g_{ a_{\wp,1} } \compos g_{ a_{\wp,2} } \compos \cdots \compos g_{ a_{\wp,t} } (\mu_{\II_A}(I_i)) = \mu_{\II_A}(I_i)$,
and then $k_{\wp} \mu_{\II_A}(I_i) = \mu_{\II_A}(I_i) k_{\wp}$. It follows that
\[ b \mu_{\II_A}(I_i) = \mu_{\II_A}(I_i) b. \]
By using (\ref{formula: thm-3}), we obtain
\[ \tilde{T}(a.f.b) = \sum\nolimits_{i} \homo(a) b_i \mu_{\II_A}(I_i)  b
   = \homo(a) \bigg( \sum\nolimits_{i}  b_i \mu_{\II_A}(I_i) \bigg) b
   = a. \tilde{T}(f).b. \]
One can prove that $\tilde{T}$ is $\FF$-linear. Therefore, $\tilde{T}$ is an $(A,B)$-homomorphism.

Second, by the definition of $\tilde{T}$, we immediately obtain
\[ \tilde{T}(\mathbf{1}_{\II_A}) = \tilde{T}(1_B\mathbf{1}_{\II_A}) = \mu(\II_A)1_B, \]
which admits \ref{H1}.

Third, for any $\displaystyle (f_t)_{1 \=< t \=< 2^{\dimA}}
  = \Big(\sum\nolimits_{i} b_{ti} \mathbf{1}_{I_i} \Big)_{1 \=< t \=< 2^{\dimA}}
  \in \bfS_{\homo}(\II_A)^{\oplus_p 2^{\dimA}}$, we have
\begin{align*}
  & \ave( \tilde{T}^{\oplus 2^{\dimA}}((f_t)_{1 \=< t \=< 2^{\dimA}}) ) \\
=~& \ave((\tilde{T}(f_t))_{1 \=< t \=< 2^{\dimA}})
= \ave\Big(\Big(
   \sum\nolimits_{i} \mu_{\II_A}(I_i) b_{ti}
  \Big)_{1 \=< t \=< 2^{\dimA}}
  \Big) \\
=~& \sum\nolimits_{t}\frac{ \mu_{\II_A}(I_t)}{\mu_{\II_A}(\II_A)}
      \sum\nolimits_{i} \mu_{\II_A}(I_i) b_{ti}
\end{align*}
and
\begin{align*}
  & \tilde{T}(\gamma_{\xi}((f_t)_{1 \=< t \=< 2^{\dimA}}) )
= \tilde{T}
      \Big(
        \sum\nolimits_{t}
          f_t\mathbf{1}_{\II_t}
       \Big)
= \tilde{T}
      \Big(
        \sum\nolimits_{t}
            \sum\nolimits_{i} b_{ti}\mathbf{1}_{I_i}
       \Big) \\
=~& \tilde{T}
      \Big(
        \sum\nolimits_{t}
           \sum\nolimits_{i}
             \frac{\mu_{\II_A}(I_i)}{\mu_{\II_A}(\II_A)}
               b_{ti}\mathbf{1}_{I_i}
       \Big)
= \sum\nolimits_{t}
    \sum\nolimits_{i}
      \frac{\mu_{\II_A}(I_i)}{\mu_{\II_A}(\II_A)}
        \mu_{\II_A}(I_i) b_{ti}.
\end{align*}
Thus, $ \ave \compos \tilde{T}^{\oplus 2^{\dimA}} = \tilde{T} \compos \gamma_{\xi}$, i.e., \ref{H2} holds.

Therefore, $T$ is a morphism in
$\Hom_{\scrN^1_A}((\bfS_{\homo}(\II_A), \mathbf{1}_{\II_A}, \gamma_{\xi}),
(B, \mu_{\II_A}(\II_A)1_B, \ave))$, and the completion $\w{\bfS_{\homo}(\II_A)}$ of $\bfS_{\homo}(\II_A)$ induces that $\w{T}$ is a morphism in
$\Hom_{\scrA^1_A}((\bfS_{\homo}(\II_A), \mathbf{1}_{\II_A}, \gamma_{\xi}),
(B,$ $\mu_{\II_A}(\II_A)1_B, \ave))$ as required. We have completed the proof of (1).

(2) We have proved that $T$ satisfies \ref{I1} in the proof of (1), then it is clear that $\w{T}$ satisfies \ref{I1} by the completion $\w{\bfS_{\homo}(\II_A)}$ of $\bfS_{\homo}(\II_A)$.
Moreover, for each $p\>=1$ and $h \in \w{\bfS_{\homo}(\II_A)}$, $\Vert h \Vert_{B,p} 1_B$ is also a function in $\w{\bfS_{\homo}(\II_A)}$, then $\Vert h \Vert_{B,p} 1_B$ can be seen as a projective limit $\invlim \Vert h_t\Vert_{B,p} 1_B$ of some Cauchy sequence $\Vert h_t\Vert_{B,p} 1_B$, where, for each $t$, we have $h_t \in E_{u(t)}$, and so $\Vert h_t \Vert_{B,p} 1_B$ can be written as a finite sum
\[ \Vert h_t \Vert_{B,p} 1_B =  \Vert h_t \Vert_{B,p} \mathbf{1}_{\II_A}
  = \sum_{i\in J} y_{ti} \mathbf{1}_{I_i}
  \text{~with~} y_{ti} \in\RR^{\>=0}, \]
where $J$ is a finite index set. Thus,
\begin{align}\label{formula: thm-3 I}
  T(\Vert h_t \Vert_{B,p} 1_B )
 = T|_{E_{u(t)}}(\Vert h_t \Vert_{B,p} 1_B )
 = \sum_{i\in J} \mu_{\II_A}(I_i)y_{ti} 1_B
 = \bigg(\sum_{i\in J} \mu_{\II_A}(I_i)y_{ti}\bigg) 1_B
\end{align}
which is of the form $\omega_t 1_B$ lying in $\RR^{\>=0} 1_B$.

Notice that the norm $\Vert T|_{E_0} \Vert$ of $T|_{E_0}$, as an $\FF$-linear map, is
\[ \sup_{ { f\in E_0 \cong B  \atop \Vert f \Vert_{E_0} = 1 }  } \Vert T|_{E_0}(f) \Vert
= \Vert T|_{E_0}(\mathbf{1}_{\II_A}) \Vert_{B,p} = (\mu_{\II_A}(\II_A)^p)^{\frac{1}{p}}
= \mu_{\II_A}(\II_A), \]
and, for each $u\in\NN$, $\Vert T|_{E_u} \Vert = \mu_{\II_A}(\II_A)$ yields
\begin{align}
  \Vert T|_{E_{u+1}} \Vert
& = \sup_{ { f\in E_{u+1} \atop \Vert f \Vert_{E_{u+1}} = 1 } }
    \Vert T|_{E_{u+1}}(f) \Vert_{B,p} \nonumber \\
& = \sup_{
        \sum\limits_{i=1}^{2^{\dimA}}
      \left(\left(
        \frac{\mu_{\II_A}(\II_i)}{\mu_{\II_A}(\II_A)}
      \right)^p \Vert f_i \Vert_{\bfS_{\homo}(\II_A)}^p\right)^{\frac{1}{p}} = 1 }
    \bigg\Vert T|_{E_{u}}
    \bigg(
      \sum_{i=1}^{2^{\dimA}}
        \frac{\mu_{\II_A}(\II_i)}{\mu_{\II_A}(\II_A)}
          (f_i\mathbf{1}_{\II_i})
    \bigg)
    \bigg\Vert_{B,p}
\mathop{=}\limits^{\clubsuit} \Vert T|_{E_u} \Vert,  \label{formula: thm-3 II}
\end{align}
where $\clubsuit$ is given by the formula
\[ \sum_{i=1}^{2^{\dimA}}
        \bigg(\frac{\mu_{\II_A}(\II_i)}{\mu_{\II_A}(\II_A)}\bigg)^p
          \Vert f_i \Vert_{\bfS_{\homo}(\II_A)}^p
 = \bigg\Vert
     \sum_{i=1}^{2^{\dimA}}
        \frac{\mu_{\II_A}(\II_i)}{\mu_{\II_A}(\II_A)}
          (f_i\mathbf{1}_{\II_i})
   \bigg\Vert_{E_{u+1}}\]
that is given by the definition of the norm $\Vert\cdot\Vert_{E_{u+1}} : E_{u+1} {\xymatrix@C=2cm{\ar[r]^{\text{Lemma \ref{lemm:Eu}}}_{\cong} & }} E_{u}^{\oplus_p 2^{\dimA}} \to \RR^{\>=0}$ shown in Lemma \ref{lemm:dir sum},
Then (\ref{formula: thm-3 II}) shows
\[ \mu_{\II_A}(\II_A) = \Vert T|_{E_0} \Vert = \Vert T|_{E_1} \Vert = \cdots = \Vert T|_{E_u} \Vert = \cdots \]
by induction. Thus, $\Vert \w{T} \Vert = \mathop{\invlim}\limits_{u} \Vert T|_{E_u} \Vert = \mu_{\II_A}$, i.e., the morphism $\w{T}$, as an $\FF$-linear map defined on $\bfS_{\homo}(\II_A)$, is bounded.
It follows that
\begin{align}\label{formula: thm-3 III}
  \w{T}(\invlim~h_t) = \invlim~ \w{T}(h_t)
\end{align}
holds for all Cauchy sequences $\{h_t\}_{t\in\NN}$.
Then (\ref{formula: thm-3 I}) yields
\[ T(\Vert h \Vert_{B,p} 1_B )
 = \mathop{\invlim}\limits_{t} T(\Vert h_t \Vert_{B,p} 1_B )
 = \mathop{\invlim}\limits_{t} \bigg(\sum_{i\in J} \mu_{\II_A}(I_i)y_{ti}\bigg) 1_B
 = (\mathop{\invlim}\limits_{t} \omega_t) 1_B \in \RR^{\>=0}1_B. \]
Of course \ref{I2} holds in the case for $p=1$. Furthermore, \ref{I3} is a direct corollary of (\ref{formula: thm-3 III}). We have completed this proof.
\end{proof}

Obviously, when $p=1$, $\FF=\RR=B$, $A=\alg(\Q_A,\vecd_A,\vecg_A,\RR,(\RR_i=\RR)_{i\in(\Q_A)_0})$ with $(\Q_A)_1=\varnothing$, all components of $\vecd_A$ is $1$, and all components of $\vecg_A$ is $\mathrm{id}_{\RR}$, then \ref{I1}, \ref{I2}, and \ref{I3} yield \ref{D1}, \ref{D2}, and \ref{D3}, respectively.
In this case, the $(A,B)$-homomorphism $\w{T}$ given in Theorem \ref{thm:3} provides a categorification of Daniell integral.

\subsection{Bochner integrations} \label{subsect:Boch int}

Let $X = (X,\mu)$ be a completed Banach space with Lebesgue measure $\mu$ and $f:\Omega\to \CC^n$ a vector-valued function. If $f$ is the limit of a sequence $\{s_u(x)\}_{u=1}^{+\infty}$ of some countable valued functions
$\displaystyle s_u = \sum_{i=1}^{+\infty} y_i \mu(I_i)$
(i.e., the sequence of some such functions whose images are countable sets),
where $\displaystyle y_i\in\CC^n$, $\displaystyle \Omega=\bigcup_{i=1}^{m_u} I_i$ is a disjoint union such that
$\displaystyle \sum_{i=1}^{+\infty} \Vert y_i \Vert \mu(I_i) < +\infty$,
and the multiple integral
$\displaystyle \lim\limits_{u \to +\infty} \int_{\Omega} \Vert f-s_u\Vert \dd\mu$
converges to zero,
then the {\defines Bochner integral} of $f$ is defined as
\[ (\Bochner)\int_{\Omega} f \dd\mu
:= \lim_{u \to +\infty}(\Bochner)\int_{\Omega} s_u \dd\mu
:= \lim_{u \to +\infty} \sum_{i=1}^{m_u} y_i \mu(I_i), \]
see \cite{B1933Integration}.

In Theorem \ref{thm:3}, take $\FF=\RR$, and we assume that $A=\alg(\Q_A,\vecd_A,\vecg_A,\EE,(\FF_i)_{i\in(\Q_A)_0})$ is a semi-simple $\RR$-algebra with $\vecd_A=(1,1,\ldots,1)$, $\vecg_A = (\mathrm{id}_{\EE}, \mathrm{id}_{\EE}, \ldots, \mathrm{id}_{\EE})$, and $\EE=\FF_i=\RR$;
$B=\alg(\Q_B,\vecd_B,\vecg_B,\EE,(\FF_j)_{j\in(\Q_B)_0})$ is a semi-simple $\RR$-algebra with
with $\vecd_B=(1,1,\ldots,1)$, $\vecg_B = (\mathrm{id}_{\EE}, \mathrm{id}_{\EE}, \ldots, \mathrm{id}_{\EE})$,
$\EE=\FF_j=\RR$; and $\homo: A \to B$ is zero.
Then $A = \RR^{\dimA}$ and $B = \RR^{\dimB}$ are Euclidean spaces, $\II_A = [c,d]^{\times\dimA}$, and the morphism $\w{T}$ describes vector valued integration.
Furthermore, if $\mu_{\II_A} = \mu_{\Lebesgue}$ is a Lebesgue measure, then
\begin{align}\label{formula: Boch int}
  \w{T}(f) = (\Bochner)\int_{\II_A} f \dd\mu \text{~for all~} f \in \w{\bfS_{\homo}(\II_A)},
\end{align}
i.e., $\w{T}(f)$ is the Bochner integral of $f$.

\subsection{Lebesgue integrations} \label{subsect:Leb int}

Keep the notations from Subsection \ref{subsect:Boch int}, if $\dimB = 1$, then, for any $f\in\w{\bfS_{\homo}(\II_A)}$, (\ref{formula: Boch int}) describes the multiple Lebesgue integral $\w{T}(f)$ of $f$.
Canonical Lebesgue integration is defined in the sense of $\dimA=\dimB=1$, and in this case,
\begin{center}
  $\bfS_{\homo}(\II_A) \cong L_1([c,d])$
\end{center}
is $L_1$-space, see \cite{Lei2023FA} and \cite[Subsection 10.1]{LLHZ2025}. Canonical Lebesgue integration is described by a unique morphism lying in $\Hom_{\scrA^1_A}( (L_1([c,d]),\mathbf{1}_{[c,d]}, \gamma_{\frac{c+d}{2}}), (\RR, 1, \ave))$,
where $\ave: \RR \oplus_1 \RR \to \RR$ sends each $(r_1,r_2) \in \RR\oplus_1 \RR$ to the average $\frac{r_1+r_2}{2}$ of $r_1$ and $r_2$.

\section{Applications II: Approximations} \label{sect:app-approx}

\def\bfX{\mathbf{X}}

Let $\bfX_0$ be an $(A,B)$-submodule of $\Func(\II_A)$ containing $\mathbf{1}_{\II_A}: \II_A \to \{1_B\}$ such that $\w{\bfX_0} \subseteq \w{\bfS_{\homo}(\II_A)}$. For any $u\in \NN$, define
\[ \bfX_u = \{ \gamma_{\xi}|_{\bfX_{u-1}}(\pmb{f}) \mid \pmb{f}
= (f_1,f_2,\ldots, f_{2^{\dimA}}) \in \bfX_{u-1}^{\oplus_p 2^{\dim A}} \}. \]
Then for any $u\in\NN$, we have $\bfX_u \subseteq \w{\bfS_{\homo}(\II_A)}$ is a normed $(A,B)$-submodule whose norm is the restriction $\Vert\cdot\Vert_{\bfX_u} = \Vert\cdot\Vert_{\w{\bfS_{\homo}(\II_A)}} |_{\bfX_u}$ of $\Vert\cdot\Vert_{\w{\bfS_{\homo}(\II_A)}}$.
Furthermore, we have
\[ \xymatrix{
  \bfX_0 \ar[r]^{\subseteq}
& \bfX_1 \ar[r]^{\subseteq}
& \cdots \ar[r]^{\subseteq}
& \bfX_u \ar[r]^{\subseteq}
& \cdots & (\subseteq \w{\bfS_{\homo}(\II_A)}).
} \]
Denote by $\bfX^{\lim} := \dirlim \bfX_u$. In this section, we show Stone--Weierstrass Theorem in $\scrA^p_{\homo}$.

\subsection{Stone--Weierstrass Approximation Theorem} \label{subsect:SWThm}

Classically, the Stone--Weierstrass Approximation Theorem states that any continuous function on a compact interval can be uniformly approximated by polynomials. Its original version can be found in \cite{Weierstrass1885}, and later, Cambridge University Press printed a new version in 2013, see \cite{Weierstrass2013}. Stone extended the works of Weierstrass in \cite{Stone1948}, proposing an algebraic approximation framework for compact spaces where the ``separation of points'' condition serves as a substitute for polynomial constraints. Here, we provide a categorical formulation of this result in the context of normed $(A,B)$-bimodules.

\begin{proposition} \label{prop:(X,1,gamma)}
The triple $(\bfX^{\lim}, \mathbf{1}_{\II_A}, \gamma_{\xi}|_{\bfX^{\lim}})$ is an object in $\scrN^p_{\homo}$.
\end{proposition}

\begin{proof}
By the definitions of $\bfX_u$ ($u\in\NN$) and $\bfX^{\lim}$, it holds that $\bfX^{\lim}$ is a normed $(A,B)$-bimodule, then \ref{Np1} holds. Here, the norm $\Vert\cdot\Vert_{\bfX^{\lim}} = \Vert\cdot\Vert_{\w{\bfS_{\homo}(\II_A)}}|_{\bfX^{\lim}}$ is induced by the inductive limit $\dirlim \Vert\cdot\Vert_{\bfX_u} = \dirlim~\Vert\cdot\Vert_{\w{\bfS_{\homo}(\II_A)}}|_{\bfX_u}$ given by $\bfX^{\lim} = \dirlim \bfX_u$.
Since $\bfX_0$ contains $\mathbf{1}_{\II_A}$, we have $\mathbf{1}_{\II_A} \in \bfX_u$ for all $u\in\NN$, and so we have $\mathbf{1}_{\II_A} \in \dirlim \bfX_{u}$. Then we obtain a homomorphism
\[ \Pc: B \cong  \mathbf{1}_{\II_A} B \longrightarrow \dirlim\bfX_u \]
which is induced by $\mathbf{1}_{\II_A} B \subseteq \bfX_0 \subseteq \bfX_u  \subseteq \bfX^{\lim}=\dirlim \bfX_u $. Thus, \ref{Np2} holds by the fact $\Pc(1_B) = \mathbf{1}_{\II_B}$.
\ref{Np3} is trivial since $\bfX^{\lim}$ is a submodule of $\w{\bfS_{\homo}(\II_A)}$. Therefore, $(\bfX^{\lim}, \mathbf{1}_{\II_A}, \gamma_{\xi}|_{\bfX^{\lim}})$ is an object in $\scrN^p_{\homo}$.
\end{proof}

The following corollary provides a categorical description of the Stone--Weierstrass Approximation Theorem.

\begin{corollary}[{Stone--Weierstrass Approximation Theorem}] \label{coro:SWThm}
\[ \sharp \Hom_{\scrA^p_{\homo}}( (\w{\bfS_{\homo}(\II_A)}, \mathbf{1}_{\II_A}, \w{\gamma}_{\xi}),
(\w{\bfX^{\lim}}, \mathbf{1}_{\II_A}, \w{\gamma_{\xi}|_{\bfX^{\lim}}} ) ) = 1. \]
\end{corollary}

\begin{sloppypar}
\begin{proof}
By Proposition \ref{prop:(X,1,gamma)}, $(\bfX^{\lim}, \mathbf{1}_{\II_A}, \gamma_{\xi}|_{\bfX^{\lim}})$ is an object in $\scrN^p_{\homo}$, then the triple $(\w{\bfX^{\lim}}, \mathbf{1}_{\II_A}, \w{\gamma_{\xi}|_{\bfX^{\lim}}} )$ induced by the completion of $\bfX^{\lim}$ is an object in $\scrA^p_{\homo}$. Thus, this statement holds by Theorem \ref{thm:1}.
\end{proof}
\end{sloppypar}

This categorical version of the Stone-Weierstrass theorem will be applied in the following subsections to power series expansions (Subsection \ref{subsect:PowerSeries}) and Fourier series expansions (Subsection \ref{subsect:FourierSeries}), demonstrating its utility in analysis.

\subsection{Power series expansion} \label{subsect:PowerSeries}

\def\Pow{\mathrm{pow}}
\def\rmT{\mathrm{T}}
\def\Ana{\mathrm{Ana}}

We assume that the following Assumption \ref{assume} holds in this subsection.

\begin{assumption} \label{assume} \rm
$A=B=\FF$, $\homo=\mathrm{id}_{\FF}$, $\II_A=[0,1]$, $\xi=\frac{1}{2}$, $\mu_{\II_A}$ be a Lebesgue measure, and $p=1$.
\end{assumption}

\begin{sloppypar}
If $\FF=\RR$, then we have $\w{{\bfS}_{\homo}(\II_A)} \cong \dirlim E_u \cong L_1([0,1])$ by \cite{Lei2023FA}.
Let $\bfX_u = \mathrm{span}_{\RR}\{ x^t \mid t\in\ZZ, -u\=< t \=< u \}$ for any $u\in\NN$, then the $\RR$-action
\[ A \times \bfX_u \to \bfX_u, ~ \bigg(r,\sum_{i=-u}^{+u} r_ux^u \bigg) \mapsto \sum_{i=-u}^{+u} rr_ux^u \]
both a left $A$-action and a right $B$-action, i.e., $\bfX_u$ is a normed $(\RR,\RR)$-bimodule in this case.
Thus, $\bfX^{\lim} \cong \w{\RR[x,x^{-1}]}$ is also a normed $(\RR,\RR)$-bimodule. Here, the norm defined on $\bfX_u$ is the restriction $\Vert\cdot\Vert_{L_1([0,1])}|_{\bfX_u}$ and the norm defined on $\bfX^{\lim}$ is the restriction $\Vert\cdot\Vert_{L_1([0,1])}|_{\bfX^{\lim}}$.
By canonical analysis, it is well-known that $\RR[x,x^{-1}]$ is dense in $L_1([0,1])$, then we have
\begin{align}\label{iso:R[x]-L1}
  \w{\RR[x,x^{-1}]} \cong L_1([0,1]),
\end{align}
and so the $(\RR,\RR)$-homomorphism
\[ H_{\Pow}: (L_1([0,1]), \mathbf{1}_{[0,1]}, \w{\gamma}_{\frac{1}{2}})
  \to (\w{\RR[x,x^{-1}]}, x^0, \w{\gamma_{\frac{1}{2}}|_{\RR[x,x^{-1}]}} ), \]
as an $\RR$-linear map, is a unique morphism in
$\Hom_{\scrA_{\mathrm{id}_{\RR}}^1}
( (L_1([0,1]), \mathbf{1}_{[0,1]}, \w{\gamma}_{\frac{1}{2}}),
(\w{\RR[x,x^{-1}]}, x^0,$ $\w{\gamma}_{\frac{1}{2}}|_{\RR[x,x^{-1}]} )
)$ by Corollary \ref{coro:SWThm}, and (\ref{iso:R[x]-L1}) yields that $H_{\Pow}$ is an $\RR$-linear isomorphism.

On the other hand, for each analytic function $f$ in $L_1([0,1])$, it has a Taylor series expansion
\[ \rmT: f(x) \mapsto \sum_{n=0}^{+\infty} \frac{x^n}{n!}\frac{\dd^n}{\dd x^n}f(0) \]
which can be viewed as a map
\[ \rmT: \Ana([0,1]) \to \w{\RR[x,x^{-1}]}, ~
   f(x) \mapsto \sum_{n=0}^{+\infty} \frac{\mathbf{1}_{[0,1]}x^n}{n!}\frac{\dd^n}{\dd x^n}f(0)\]
where $\Ana([0,1])$ is the set of all analytic functions in $L_1([0,1])$.
One can check that $\rmT$ is an $(\RR,\RR)$-homomorphism (i.e., an $\RR$-linear map) such that:
\begin{itemize}
  \item[(1)] $\Ana([0,1])$ is an $(\RR,\RR)$-bimodule with the norm $\Vert\cdot\Vert_{\Ana([0,1])} := \Vert\cdot\Vert_{L_1([0,1])}|_{\Ana([0,1])}$, i.e., \ref{Np1} holds;
  \item[(2)] $\mathbf{1}_{[0,1]} \in \Ana([0,1])$ is a function with norm $\Vert \mathbf{1}_{[0,1]}\Vert = 1$ such that the $\RR$-linear map $\Pc: \RR \to \Ana([0,1])$ induced by $\RR \mathbf{1}_{[0,1]} := \{r\mathbf{1}_{[0,1]} \mid r\in \RR\} \subseteq \Ana([0,1])$ sends $1$ to $\mathbf{1}_{[0,1]}$, i.e., \ref{Np2} holds,
  \item[(3)] $\gamma_{\frac{1}{2}}|_{\Ana([0,1])}$ satisfies \ref{Np3};
  \item[(4)] $\rmT(\mathbf{1}_{[0,1]}) = \mathbf{1}_{[0,1]} x^0 = \mathbf{1}_{[0,1]} \in \RR[x,x^{-1}] \subseteq \w{\RR[x,x^{-1}]}$, i.e., \ref{H1} holds;
  \item[(5)] $\rmT(\w{\gamma}_{\frac{1}{2}}(f_1(x),f_2(x))) = \begin{cases}
      \rmT(f_1(2x)) = \sum\limits_{n=0}^{+\infty}
        \frac{\mathbf{1}_{[0,\frac{1}{2}]}(2x)^n}{n!}\frac{\dd^n}{\dd x^n}f_1(0) ,
    & 0\=< x< \frac{1}{2}; \\
      \rmT(f_2(2x-1)) = \sum\limits_{n=0}^{+\infty}
        \frac{\mathbf{1}_{[\frac{1}{2},1]}(2x-1)^n}{n!}\frac{\dd^n}{\dd x^n}f_2(0) ,
    & \frac{1}{2} \=< x \=<1
    \end{cases}$
    $ = \w{\gamma}_{\frac{1}{2}}( \rmT(f_1(x)), \rmT(f_2(x)))$
    $ = \w{\gamma}_{\frac{1}{2}}\rmT^{\oplus 2} (f_1(x), f_2(x))$,
    i.e., \ref{H2} holds.
\end{itemize}
Therefore, $\rmT$ is a morphism in $\Hom_{\scrN_{\mathrm{id}_{\RR}}^1}
( (L_1([0,1]), \mathbf{1}_{[0,1]}, \w{\gamma}_{\frac{1}{2}}),
(\w{\RR[x,x^{-1}]}, \mathbf{1}_{[0,1]} x^0, \w{\gamma}_{\frac{1}{2}}|_{\RR[x,x^{-1}]} )
)$, and $\w{\rmT}$, the $\RR$-linear map induced by the completion $\w{\RR[x,x^{-1}]}$ of $\RR[x,x^{-1}]$,
is a morphism in $\Hom_{\scrA_{\mathrm{id}_{\RR}}^1}
( (L_1([0,1]), \mathbf{1}_{[0,1]}, \w{\gamma}_{\frac{1}{2}}),
(\w{\RR[x,x^{-1}]}, \mathbf{1}_{[0,1]}x^0, \w{\gamma}_{\frac{1}{2}}|_{\RR[x,x^{-1}]} )
)$ by Theorem \ref{thm:2}.
By the uniqueness, $\rmT = H_{\Pow}$, i.e., the morphism $ H_{\Pow}$ given by Corollary \ref{coro:SWThm} provides a categorification of power series expansions of analytic functions.
\end{sloppypar}

\subsection{Fourier series expansion} \label{subsect:FourierSeries}

\def\rme{\mathrm{e}}
\def\rmi{\mathrm{i}}
\def\Fou{\mathrm{Fou}}

Keep the notations in Assumption \ref{assume} in this subsection, and let $\FF=\CC$ and $\bfX_u = \mathrm{span}_{\CC}\{ \rme^{2t\pi\rmi x} \mid -u\=< t \=< u \}$.
Then $\bfX^{\lim} = \w{\CC[\rme^{\pm2\pi\rmi x}]}$. Notice that it is well-known that $\CC[\rme^{\pm2\pi\rmi x}]$ is a dense $\CC$-subspace of $L_1([0,1])$ in canonical analysis, we obtain
\begin{align} \label{iso:C[e power]-L1}
  \w{\CC[\rme^{\pm2\pi\rmi x}]} \cong L_1([0,1]),
\end{align}
and so the $(\RR,\RR)$-homomorphism
\[ H_{\Fou}: (L_1([0,1]), \mathbf{1}_{[0,1]}, \w{\gamma}_{\frac{1}{2}})
  \to (\w{\CC[\rme^{\pm2\pi\rmi x}]}, \mathbf{1}_{[0,1]}\rme^0,
       \w{\gamma}_{\frac{1}{2}}|_{ \w{\CC[\rme^{\pm2\pi\rmi x}]} } ), \]
as an $\RR$-linear map, is a unique morphism in
$\Hom_{\scrA_{\mathrm{id}_{\RR}}^1}
( (L_1([0,1]), \mathbf{1}_{[0,1]}, \w{\gamma}_{\frac{1}{2}}),
(\w{\CC[\rme^{\pm2\pi\rmi x}]}, \mathbf{1}_{[0,1]}\rme^0,$ $\w{\gamma}_{\frac{1}{2}}|_{ \w{\CC[\rme^{\pm2\pi\rmi x}]} } )
)$ by Corollary \ref{coro:SWThm}.
By using a method similar to Subsection \ref{subsect:PowerSeries}, (\ref{iso:C[e power]-L1}) is an $\RR$-linear isomorphism sending each analytic function $f$ lying in $L_1([0,1])$ to the trigonometric series it.
Furthermore, if $f$ satisfies the Dirichlet Condition, then $H_{\Fou}(f)$ is the Fourier series of $f$.

\section{\texorpdfstring{An example for integration in $\scrA_{\homo}^1$}{An example for integration in A1}} \label{sect:exp}

We provide an example in this section.

\begin{example} \rm \label{exp:1}
Let $(\Q,\vecd=(2,2,1))$ is the weight quiver given in Example \ref{exp:modulation w ne 0} and $ab=bc=ca=0$.
Assume $\FF=\FF_3=\RR$, $\FF_1=\FF_2=\CC$, and $\vecg = (g_a, g_b, g_c) = (\mathrm{id}_{\CC}, \mathrm{id}_{\RR}, \mathrm{id}_{\RR})$. Then the modulation corresponded by the tensor ring $\alg(\Q,\vecd,\vecg,\CC,(\RR)_{i\in\Q_0})$ is
\begin{figure}[H]
\centering
\begin{tikzpicture} 
\draw (-1.73,-1) node{$\CC$} (0,2) node{$\CC$} (1.73,-1) node{$\RR$};
\draw [rotate=  0][->] (-1.43,-1) -- (1.43,-1);
\draw [rotate=120][->] (-1.43,-1) -- (1.43,-1);
\draw [rotate=240][->] (-1.43,-1) -- (1.43,-1);
\draw [rotate= 90](2.5,0) node{$d_1=2$};
\draw [rotate=210](2.5,0) node{\rotatebox{120}{$d_2=2$}};
\draw [rotate=330](2.5,0) node{\rotatebox{240}{$d_3=1$}};
\draw [rotate=30+120] (0.7,0) node{\rotatebox{ 60}{$A_a$}};
\draw [rotate=30+120] (1.3,0)
  node{\rotatebox{ 60}{$\CC \otimes_{\CC} \CC^{\mathrm{id}_{\CC}} \cong \CC $}};
\draw [rotate=30+  0] (0.7,0) node{\rotatebox{-60}{$A_c$}};
\draw [rotate=30+  0] (1.3,0)
  node{\rotatebox{-60}{$\RR \otimes_{\RR} \CC^{\mathrm{id}_{\RR}} \cong \CC $}};
\draw [rotate=30+240] (0.7,0) node{\rotatebox{  0}{$A_b$}};
\draw [rotate=30+240] (1.3,0)
  node{\rotatebox{  0}{$\CC \otimes_{\RR} \RR^{\mathrm{id}_{\RR}} \cong \CC $}};
\end{tikzpicture}
\end{figure}
\noindent
modulo $\I = (A_a\otimes_{\CC} A_b) \oplus (A_b\otimes_{\RR} A_c) \oplus (A_c \otimes_{\CC} A_a)$. Thus,
\begin{align*}
  A := A/\I  =~& \CC\e_1 + \CC\e_2 + \RR\e_3 + A_a + A_b + A_c + \I \\
  =~& \RR\e_1 + {\color{blue}\RR\rmi\e_1} + \RR\e_2 + {\color{blue}\RR\rmi\e_2} + \RR\e_3 +  \\
    & \RR a + {\color{red}\RR\rmi a} + \RR b + {\color{red}\RR\rmi b} + \RR c + {\color{red}\RR\rmi c} + \I
\end{align*}
is a finite-dimensional $\RR$-algebra whose dimension is $11$. One can check that $\e_1$, $\e_2$, $\e_3$ is a completed primitive orthogonal idempotent set of $A$ and $\rad A = \RR a + \RR\rmi a + \RR b + \RR\rmi b + \RR c + \RR\rmi c + \I$, it follows that the bound quiver $(\Q_A,\I_A)$ of $A$ is given by the quiver $\Q_A$ shown in \Pic \ref{fig:quiver in exp-1}
\begin{figure}[htbp]
\centering
\begin{tikzpicture} 
\draw (-1.73,-1) node{$2$} (0,2) node{$1$} (1.73,-1) node{$3$};
\draw [rotate=  0][->] (-1.43,-1) -- (1.43,-1);
\draw [rotate=120][->] (-1.43,-1) -- (1.43,-1);
\draw [rotate=240][->] (-1.43,-1) -- (1.43,-1);
\draw [rotate= 10][->][red] ( 1.73,-1) arc(-30:70:2);
\draw [rotate=130][->][red] ( 1.73,-1) arc(-30:70:2);
\draw [rotate=250][->][red] ( 1.73,-1) arc(-30:70:2);
\draw [rotate=-3+ 90][rotate around={8:(2.2,0)}][blue] [->] (2.2,0) arc(180:520:0.75);
\draw [rotate=-3+210][rotate around={8:(2.2,0)}][blue] [->] (2.2,0) arc(180:520:0.75);
\draw [rotate=30+120] (1.13,0) node{$a$};
\draw [rotate=30+240] (1.13,0) node{$b$};
\draw [rotate=30+360] (1.13,0) node{$c$};
\draw [rotate=30+120][red] (2.23,0) node{$a'$};
\draw [rotate=30+240][red] (2.23,0) node{$b'$};
\draw [rotate=30+360][red] (2.23,0) node{$c'$};
\draw [rotate= 90][blue] (3.9,0) node{$x_1$};
\draw [rotate=210][blue] (3.9,0) node{$x_2$};
\end{tikzpicture}
\caption{Quiver $\Q_A$}
\label{fig:quiver in exp-1}
\end{figure}
and the ideal $\I_A$ defined as
\begin{center}
  $\I_A = \langle x_1^2+\e_1, x_2^2+\e_2, x_1a'+a, a'-x_1a, x_2b'+b, b'-x_2b, c'x_1+c, c'-cx_1$,

   $a'x_2+a, a'-ax_2, ab, bc, ca, a'b', b'c', c'a', ab', bc', ca', a'b, b'c, c'a \rangle$.
\end{center}
Here, $a'$, $b'$, $c'$, $x_1$, and $x_2$ are corresponded by $\rmi a$, $\rmi b$, $\rmi c$, $\rmi\e_1$, and $\rmi\e_2$, respectively. Accordingly, \Pic \ref{fig:quiver in exp-2} is the modulation of $\Q_A$ corresponded by $A$.
\begin{figure}[htbp]
\centering
\begin{tikzpicture} 
\draw (-1.73,-1) node{$\RR$} (0,2) node{$\RR$} (1.73,-1) node{$\RR$};
\draw [rotate=  0][->] (-1.43,-1) -- (1.43,-1);
\draw [rotate=120][->] (-1.43,-1) -- (1.43,-1);
\draw [rotate=240][->] (-1.43,-1) -- (1.43,-1);
\draw [rotate= 10][->][red] ( 1.73,-1) arc(-30:70:2);
\draw [rotate=130][->][red] ( 1.73,-1) arc(-30:70:2);
\draw [rotate=250][->][red] ( 1.73,-1) arc(-30:70:2);
\draw [rotate=-3+ 90][rotate around={8:(2.2,0)}][blue] [->] (2.2,0) arc(180:520:0.75);
\draw [rotate=-3+210][rotate around={8:(2.2,0)}][blue] [->] (2.2,0) arc(180:520:0.75);
\draw [rotate=30+120] (1.25,0) node{\rotatebox{60}{$A_a\cong \RR$} };
\draw [rotate=30+240] (1.25,0) node{$A_b\cong \RR$};
\draw [rotate=30+360] (1.25,0) node{\rotatebox{-60}{$A_c\cong \RR$}};
\draw [rotate=30+120][red] (2.23,0) node{\rotatebox{60}{$A_{a'}\cong \RR$}};
\draw [rotate=30+240][red] (2.23,0) node{$A_{b'}\cong \RR$};
\draw [rotate=30+360][red] (2.23,0) node{\rotatebox{-60}{$A_{c'}\cong \RR$}};
\draw [rotate= 90][blue] (3.9,0) node{$A_{x_1}\cong \RR$};
\draw [rotate=210][blue] (3.9,0) node{\rotatebox{-45}{$A_{x_2}\cong \RR$ }};
\end{tikzpicture}
\caption{Quiver $\Q_A$}
\label{fig:quiver in exp-2}
\end{figure}
Here, for any arrow $\alpha \in (\Q_A)_1$, we have $A_{\alpha} = \RR \otimes_{\RR\cap\RR} \RR \cong \RR$.
Next, let $B = \kk\Q_B/\I_B$ be given by the quiver $\Q_B:=\Q$ and the admissible ideal $\I_B=\langle ab,bc,ca\rangle$.
Then $A/\J$ $(\cong B)$ induced an epimorphism
\[ \homo: A \to B, ~ x\mapsto x + \J, \]
where $\J=\langle x_1+\I_A, x_2+\I_A, a'+\I_A, b'+\I_A, c'+\I_A \rangle$.
Consider the restriction \[\homo|_{\II_A = [0,1]^{\times 11}}: \II_A \to B\] which is a function lying in the normed $(A,B)$-bimodue $\w{\bfS_{\homo}(\II_A)}$, and the $(A,B)$-homomorphism $\w{T}$ sends it to its integration
\begin{align*}
   & (\scrA_{\homo}^1)\int_{[0,1]^{\times 11}}\homo|_{[0,1]^{\times 11}} \dd\mu_{\II_A} \\
 =~& \sum_{i\in\{1,2,3\}}(\scrA_{\homo}^1)\int_{[0,1]^{\times 11}} r_{\e_i}(\e_i+\J) \dd\mu_{\II_A} \\
   &+ \sum_{\alpha\in\{a,b,c\}}(\scrA_{\homo}^1)\int_{[0,1]^{\times 11}} r_{\alpha}(\alpha+\J) \dd\mu_{\II_A} \\
 =~& \frac{1}{2}(1_B+a+b+c) + \J \in B
\end{align*}
in the sense of the category $\scrA_{\homo}^1$.
Here, $\mu_{\II_A}$ is a Lebesgue measure $\mu$ and each summand can be viewed as a Lebesgue integration
\begin{center}
$\displaystyle (\scrA_{\homo}^1)\int_{[0,1]^{\times 11}} k \dd\mu_{\II_A}
= \bigg( (\Lebesgue)\int_0^1\dd\mu \bigg)^{n-1} \bigg( (\Lebesgue)\int_0^1 k \dd k \bigg) = \frac{1}{2}$
\end{center}
in the sense of the $\RR$-linear isomorphism $\RR b_i \cong \RR$ ($b_i\in \{b_i \mid 1\=< i \=< 6\} := \{\e_1+\J,\e_2+\J,\e_3+\J, a+\J, b+\J , c+\J\} = \basis{B}$, and one can check that this isomorphism is also an $(A,B)$-isomorphism since $\RR(b_i+\J)$ is a normed $(A,B)$-bimodule).
Moreover, if we do not use the $(A, B)$-linearity of $\displaystyle \w{T} = (\scrA_{\homo}^1)\int_{[0,1]^{\times 11}}(\cdot)\dd\mu_{\II_A}$, then
\[
   (\scrA_{\homo}^1)\int_{[0,1]^{\times 11}}\homo|_{[0,1]^{\times 11}} \dd\mu_{\II_A}
 = (\Bochner){\iint\cdots\int}_{[0,1]^{\times 11}} (A\to A/\J) \dd\mu
\]
is a Bochner integration.
\end{example}


%
%
%

\section*{Funding}

\addcontentsline{toc}{section}{Funding}

Yu-Zhe Liu is supported by
the National Natural Science Foundation of China (Grant No. 12401042, 12171207),
Guizhou Provincial Basic Research Program (Natural Science) (Grant Nos. ZD[2025]085 and ZK[2024]YiBan066)
and Scientific Research Foundation of Guizhou University (Grant Nos. [2022]53, [2022]65, [2023]16).

\section*{Acknowledgements}

I am greatly indebted to Shengda Liu and Mingzhi Sheng for helpful suggestions.

\addcontentsline{toc}{section}{Acknowledgements}


\addcontentsline{toc}{section}{References}

   \bibliographystyle{abbrv} 

 \bibliography{referLiu20250515}


%
%
%
%
%
%

\end{document}